\documentclass[10pt]{article}%
\usepackage{a4}%
\usepackage[centertags]{amsmath}%
\usepackage{eucal,dsfont,accents,bbm,amsfonts,amssymb,amsthm,amsopn}%
\usepackage{dsfont}
\usepackage[usenames]{color}
\definecolor{citegreen}{rgb}{0,0.6,0}
\definecolor{refred}{rgb}{0.8,0,0}
\usepackage[colorlinks, citecolor=citegreen, linkcolor=refred]{hyperref}
\title{Ricci flow coupled with harmonic map flow\\
Flot de Ricci coupl\'e avec le flot harmonique}%
\author{Reto M\"{u}ller}
\date{}
\providecommand{\abs}[1]{\lvert #1\rvert}
\providecommand{\Abs}[1]{\left\lvert #1\right\rvert}
\providecommand{\scal}[1]{\langle #1\rangle}
\providecommand{\Scal}[1]{\left\langle #1\right\rangle}

\DeclareMathOperator{\grad}{grad}
\DeclareMathOperator{\Hess}{Hess}
\DeclareMathOperator{\vol}{vol}
\DeclareMathOperator{\Sym}{Sym}
\DeclareMathOperator{\id}{id}
\DeclareMathOperator{\tr}{tr}

\newcommand{\RR}{\mathbb{R}}
\newcommand{\NN}{\mathbb{N}}
\newcommand{\eps}{\varepsilon}
\newcommand{\Lap}{\triangle}
\newcommand{\D}{\nabla}
\newcommand{\hN}{\!\phantom{l}^N\!}
\newcommand{\sF}{\mathcal{F}}
\newcommand{\sW}{\mathcal{W}}
\newcommand{\sL}{\mathcal{L}}
\newcommand{\sS}{\mathcal{S}}
\newcommand{\sT}{\mathcal{T}}
\newcommand{\sD}{\mathcal{D}}

\newcommand{\Lexp}{\mathcal{L}_b\mathrm{exp}_{x_k}^{\tau_1}}

\newcommand{\dt}{\tfrac{\partial}{\partial t}}
\newcommand{\dtau}{\tfrac{\partial}{\partial \tau}}
\newcommand{\dl}{\tfrac{\partial}{\partial \lambda}}

\newcommand{\Rm}[1]{
  \def\arg{#1}
  \ifx\arg\empty
    \mathrm{Rm}
  \else
    \mathop{\mathrm{Rm}}(#1)
  \fi
}

\newcommand{\NRm}[1]{
  \def\arg{#1}
  \ifx\arg\empty
    \hN\mathrm{Rm}
  \else
    \mathop{\hN\mathrm{Rm}}(#1)
  \fi
}

\newcommand{\Rc}[1]{
  \def\arg{#1}
  \ifx\arg\empty
    \mathrm{Rc}
  \else
    \mathop{\mathrm{Rc}}(#1)
  \fi
}

\def\Xint#1{\mathchoice
{\XXint\displaystyle\textstyle{#1}}
{\XXint\textstyle\scriptstyle{#1}}
{\XXint\scriptstyle\scriptscriptstyle{#1}}
{\XXint\scriptscriptstyle\scriptscriptstyle{#1}}
\!\int}
\def\XXint#1#2#3{{\setbox0=\hbox{$#1{#2#3}{\int}$}
\vcenter{\hbox{$#2#3$}}\kern-.5\wd0}}
\def\dashint{\Xint-}
\newtheoremstyle{break}
  {12pt}
  {16pt}
  {\itshape}
  {}
  {\bfseries}
  {}
  {\newline}
  {\thmname{#1}\thmnumber{ #2}\thmnote{ \normalfont{(#3)}}}

\theoremstyle{definition}%
\theoremstyle{remark}%
\newtheorem*{rem}{Remark}
\theoremstyle{break}%
\newtheorem{lemma}{Lemma}[section]
\newtheorem{prop}[lemma]{Proposition}
\newtheorem{thm}[lemma]{Theorem}
\newtheorem{cor}[lemma]{Corollary}
\newtheorem{defn}[lemma]{Definition}
\numberwithin{equation}{section}%
\begin{document}%
\maketitle%
\pagenumbering{arabic}%
\begin{abstract}
We investigate a coupled system of the Ricci flow on a closed
manifold $M$ with the harmonic map flow of a map $\phi$ from $M$ to
some closed target manifold $N$,
\begin{equation*}
\dt g = -2\Rc{} + 2\alpha \D\phi \otimes \D\phi,\qquad \dt \phi =
\tau_g \phi,
\end{equation*}
where $\alpha$ is a (possibly time-dependent) positive coupling
constant. Surprisingly, the coupled system may be less singular than
the Ricci flow or the harmonic map flow alone. In particular, we can
always rule out energy concentration of $\phi$ a-priori by choosing
$\alpha$ large enough. Moreover, it suffices to bound the curvature
of $(M,g(t))$ to also obtain control of $\phi$ and all its
derivatives if $\alpha \geq \underaccent{\bar}\alpha>0$. Besides
these new phenomena, the flow shares many good properties with the
Ricci flow. In particular, we can derive the monotonicity of an
\emph{energy}, an \emph{entropy} and a \emph{reduced volume}
functional. We then apply these monotonicity results to rule out
non-trivial breathers and geometric collapsing at finite times.\\

\begin{center}\textbf{R\'esum\'e}\end{center}
Nous \'etudions un syst\`eme d'\'equations consistant en un couplage 
entre le flot de Ricci et le flot harmonique d'une fonction $\phi$ 
allant de $M$ dans une vari\'et\'e cible $N$,
\begin{equation*}
\dt g = -2\Rc{} + 2\alpha \D\phi \otimes \D\phi,\qquad \dt \phi =
\tau_g \phi,
\end{equation*}
o\`u $\alpha$ est une constante de couplage strictement positive (et 
pouvant d\'ependre du temps). De mani\`ere surprenante, ce syst\`eme 
coupl\'e peut \^etre moins singulier que le flot de Ricci ou le flot 
harmonique si ceux-ci sont consid\'er\'es de mani\`ere isol\'ee. En 
particulier, on 
peut toujours montrer que la fonction $\phi$ ne se concentre pas le long 
de ce syst\`eme \`a condition de prendre $\alpha$ assez grand. De plus, 
il est suffisant de borner la courbure de $(M,g(t))$ le long du flot 
pour obtenir le contr\^ole de $\phi$ et de toutes ses d\'eriv\'ees si  
$\alpha \geq \underaccent{\bar}\alpha>0$. A part ces ph\'enom\`enes 
nouveaux, ce flot poss\`ede certaines propri\'et\'es analogues \`a 
celles du flot de Ricci. En particulier, il est possible de montrer la 
monotonie d'une \emph{\'energie}, d'une \emph{entropie} et d'une 
fonctionnelle \emph{volume r\'eduit}. On utilise la monotonie de ces 
quantit\'es pour montrer l'absence de solutions en "accord\'eon" et 
l'absence  d'effondrement en temps fini le long du flot.
\end{abstract}
\newpage

\section{Introduction and main results}
Let $(M^m\!,g)$ and $(N^n\!,\gamma)$ be smooth Riemannian manifolds
without boundary. According to Nash's embedding theorem
\cite{Nash:embedding} we can assume that $N$ is isometrically
embedded into Euclidean space $(N^n,\gamma) \hookrightarrow
\RR^d$ for a sufficiently large $d$. If $e_N:N \to \RR^d$ denotes
this embedding, we identify maps $\phi:M\to N$ with $e_N \circ
\phi:M\to\RR^d$, such maps may thus be written as $\phi
=(\phi^\lambda)_{1\leq\lambda\leq d}$. Harmonic maps $\phi:M\to N$
are critical points of the energy functional
\begin{equation}\label{0.eq1}
E(\phi) = \int_M \abs{\D\phi}^2 dV.
\end{equation}
Here, $\abs{\D\phi}^2 := 2e(\phi) =
g^{ij}\D_i\phi^\lambda\D_j\phi^\lambda$ denotes the local energy
density, where we use the convention that repeated Latin indices are
summed over from $1$ to $m$ and repeated Greek indices are summed
over from $1$ to $d$. We often drop the summation indices for $\phi$
when clear from the context. Harmonic maps generalize the concept of
harmonic functions and in particular include closed geodesics and minimal
surfaces.%
\newline

To study the existence of a harmonic map $\phi$ homotopic to a given
map $\phi_0:M\to N$, Eells and Sampson \cite{EellsSampson} proposed
to study the $L^2$-gradient flow of the energy functional
(\ref{0.eq1}),
\begin{equation}\label{0.eq2}
\dt \phi = \tau_g\phi,\qquad \phi(0)=\phi_0,
\end{equation}
where $\tau_g\phi$ denotes the intrinsic Laplacian of $\phi$, often
called the tension field of $\phi$. They proved that if $N$ has
non-positive sectional curvature there always exists a unique,
global, smooth solution of (\ref{0.eq2}) which converges smoothly to
a harmonic map $\phi_\infty:M\to N$ homotopic to $\phi_0$ as
$t\to\infty$ suitably. On the other hand, without an assumption on
the curvature of $N$, the solution might blow up in finite or
infinite time. Comprehensive surveys about harmonic maps and the harmonic map
flow are given in Eells-Lemaire \cite{EellsLemaireI,EellsLemaireII},
Jost \cite{Jost:Riemannian} and Struwe \cite{Struwe:EvProblems}. The
harmonic map flow was the first appearance of a nonlinear heat flow
in Riemannian geometry. Today, geometric heat flows have become an
intensely studied topic in
geometric analysis.%
\newline

Another fundamental problem in differential geometry is to find
canonical metrics on Riemannian manifolds, for example metrics with 
constant curvature in some sense. Using the idea of evolving an object to 
such an ideal state by a nonlinear heat flow, Richard Hamilton
\cite{Hamilton:3folds} introduced the Ricci flow in 1982. His
idea was to smooth out irregularities of the curvature by evolving a
given Riemannian metric $g$ on a manifold $M$ with respect to the
nonlinear weakly parabolic equation
\begin{equation}\label{0.eq3}
\dt g = -2 \Rc{},\qquad g(0)=g_0,
\end{equation}
where $\Rc{}$ denotes the Ricci curvature of $(M,g)$.
Strictly speaking, the Ricci flow is not
the gradient flow of a functional $\sF(g)=\int_M F(\partial^2
g,\partial g,g)dV$, but in 2002, Perelman \cite{Perelman:entropy}
showed that it is gradient-like nevertheless. He presented a
new functional which may be regarded as an improved version of the
Einstein-Hilbert functional $E(g)=\int_M R\,dV$, namely
\begin{equation}\label{0.eq4}
\sF(g,f) := \int_M \Big(R + \abs{\D f}^2 \Big)e^{-f} dV.
\end{equation}
The Ricci flow can be interpreted as the gradient flow of $\sF$ modulo
a pull-back by a family of diffeomorphisms. Hamilton's Ricci flow has a 
successful history. Most importanty, Perelman's
work \cite{Perelman:entropy,Perelman:surgery} led to a completion of
Hamilton's program \cite{Hamilton:survey} and a complete proof 
of Thurston's geometrization conjecture \cite{Thurston} and (using 
a finite extinction result from Perelman \cite{Perelman:extinction} 
or Colding and Minicozzi \cite{ColdingMinicozziI,ColdingMinicozziII}) 
of the Poincar\'{e} conjecture \cite{Poincare}. Introductory surveys on 
the Ricci flow and Perelman's functionals
can be found in the books by Chow and Knopf \cite{RF:intro}, Chow,
Lu and Ni \cite{ChowLuNi}, M\"{u}ller \cite{Muller:Harnack} and Topping
\cite{Topping:Lectures}. More advanced explanations of Perelman's
proof of the two conjectures are given in Cao and Zhu \cite{CaoZhu}
Chow et al. \cite{RF:TAI,RF:TAII}, Kleiner and Lott
\cite{KleinerLott} and Morgan and Tian
\cite{MorganTian,MT:completion}. A good survey on Perelman's work is
also given in Tao \cite{Tao:Perelman}.%
\newline

The goal of this article is to study a coupled system of the two
flows (\ref{0.eq2}) and (\ref{0.eq3}). Again, we let $(M^m\!,g)$ and
$(N^n\!,\gamma)$ be smooth manifolds without boundary and with
$(N^n\!,\gamma) \hookrightarrow \RR^d$. Throughout this article, we
will assume in addition that $M$ and $N$ are compact, hence closed.
However, many of our results hold for more general manifolds.%
\newline

Let $g(t)$ be a family of Riemannian metrics on $M$ and $\phi(t)$ a
family of smooth maps from $M$ to $N$. We call $(g(t),\phi(t))_{t
\in [0,T)}$ a solution to the coupled system of Ricci flow and
harmonic map heat flow with coupling constant $\alpha(t)$, the
$(RH)_\alpha$ flow for short, if it satisfies
\begin{equation}
\left\{\begin{aligned}\dt g &= -2\Rc{} + 2\alpha \D\phi \otimes \D\phi,\\%
\dt \phi &= \tau_g \phi. \end{aligned}\right. \tag*{$(RH)_\alpha$}
\end{equation}
Here, $\tau_g \phi$ denotes the tension field of the map $\phi$ with
respect to the evolving metric $g$, and $\alpha(t)\geq 0$ denotes a
(time-dependent) coupling constant. Finally, $\D\phi\otimes\D\phi$
has the components $(\D\phi\otimes\D\phi)_{ij}
=\D_i\phi^\lambda\D_j\phi^\lambda$. In particular, $\abs{\D\phi}^2$
as defined above is the trace of $\D\phi\otimes\D\phi$ with respect
to $g$.%
\newline

The special case where $N\subseteq\RR$ and $\alpha\equiv 2$ was studied by List \cite{List:diss}, his motivation coming from general relativity and the study of Einstein vacuum equations. Moreover, List's flow also arises as the Ricci flow of a warped product, see \cite[Lemma A.3]{Muller:diss}. After completion of this work, we learned that another special case of $(RH)_\alpha$ with $N\subseteq SL(k\RR)/SO(k)$ arises in the study of the long-time behaviour of certain Type III Ricci flows, see Lott \cite{Lott:TypeIII} and a recent paper of Williams \cite{Williams} for details and explicit examples.%
\newline

The paper is organized as follows. In order to get a feeling for the
flow, we first study explicit examples of solutions of $(RH)_\alpha$ 
as well as soliton solutions which are generalized fixed points modulo
diffeomorphisms and scaling. The stationary solutions of
$(RH)_\alpha$ satisfy $\Rc{} = \alpha \D \phi \otimes \D \phi$,
where $\phi$ is a harmonic map. To prevent
$(M,g(t))$ from shrinking to a point or blowing up, it is convenient
to introduce a volume-preserving version of the flow.%
\newline

In Section 3, we prove that for constant coupling functions 
$\alpha(t) \equiv \alpha > 0$ the $(RH)_\alpha$ flow can be interpreted 
as a gradient flow for an energy functional $\sF_\alpha(g,\phi,f)$ 
modified by a family of diffeomorphisms generated by $\D f$. If 
$(g(t),\phi(t))$ solves $(RH)_\alpha$ and $e^{-f}$ is a
solution to the adjoint heat equation under the flow, then
$\sF_\alpha$ is non-decreasing and constant if and only if
$(g(t),\phi(t))$ is a steady gradient soliton. In the more general
case where $\alpha(t)$ is a positive function, the monotonicity
result still holds whenever $\alpha(t)$ is non-increasing. This section
is based on techniques of Perelman \cite[Section 1]{Perelman:entropy} 
for the Ricci flow.%
\newline

In the fourth section, we prove short-time existence for the flow
using again a method from Ricci flow theory known as DeTurck's trick
(cf.~\cite{DeTurck:trick}), i.e.~we transform the weakly parabolic
system $(RH)_\alpha$ into a strictly parabolic one by pushing it
forward with a family of diffeomorphisms. Moreover, we
compute the evolution equations for the Ricci and scalar curvature,
the gradient of $\phi$ and combinations thereof. In particular, the
evolution equations for the symmetric tensor $S_{ij}:= R_{ij}-\alpha
\D_i\phi\D_j\phi$ and its trace $S=R-\alpha\abs{\D\phi}^2$ will be
very useful.%
\newline

In Section 5, we study first consequences of the evolution equations
for the existence or non-existence of certain types of
singularities. Using the maximum principle, we show that $\min_{x\in
M}S(x,t)$ is non-decreasing along the flow. This has the rather
surprising consequence that if $\abs{\D\phi}^2(x_k,t_k)\to\infty$
for $t_k\nearrow T$, then $R(x_k,t_k)$ blows up as well,
i.e.~$g(t_k)$ must become singular as $t_k\nearrow T$. Conversely,
if $\abs{\Rm{}}$ stays bounded along the flow, $\abs{\D\phi}^2$ must
stay bounded, too. This leads to the conjecture that a uniform
Riemann-bound is enough to conclude long-time existence. This
conjecture is proved in Section 6. To this end, we first compute
estimates for the Riemannian curvature tensor, its derivatives and
the higher derivatives of $\phi$ and then follow Bando's
\cite{Bando:estimates} and Shi's \cite{Shi:complete} results for the
Ricci flow to derive interior-in-time gradient estimates.%
\newline

In Section 7, we introduce an entropy functional $\sW_\alpha(g,\phi,f,\tau)$
which corresponds to Perelman's shrinker entropy for the Ricci flow
\cite[Section 3]{Perelman:entropy}. Here $\tau=T-t$ denotes a
backwards time. For $\alpha(t)\equiv\alpha>0$, the entropy functional 
is non-decreasing and constant exactly on shrinking solitons. Again, the entropy is
monotone if we allow non-increasing positive coupling functions
$\alpha(t)$ instead of constant ones. Using $\sF_\alpha$ and
$\sW_\alpha$ we can exclude nontrivial breathers, i.e.~we show that
a breather has to be a gradient soliton. In the case of a steady or
expanding breather the result is even stronger, namely we can show
that $\phi(t)$ has to be harmonic in these cases for all $t$.
\newline

Finally in the last section, we state the monotonicity of a backwards 
reduced volume quantity for the $(RH)_\alpha$ flow
with positive non-increasing $\alpha(t)$. This follows from our more
general result from \cite{Muller:MonotoneVolumes}. We apply this
monotonicity to deduce a local non-collapsing theorem.%
\newline

In the appendix, we collect the commutator identities on bundles like 
$T^*M\otimes \phi^*TN$, which we need for the evolution equations in 
Section 4 and 6.%
\newline

This article originates from the authors PhD thesis \cite{Muller:diss} 
from 2009, where some of the proofs and computations are carried out in more details.
The author likes to thank Klaus Ecker, Robert Haslhofer, Gerhard
Huisken, Tom Ilmanen, Peter Topping and in particular Michael Struwe
for stimulating discussions and valuable remarks and suggestions
while studying this new flow. Moreover, he thanks the Swiss National
Science Foundation that partially supported his research and Zindine Djadli
who translated the abstract into flawless French.%

\section{Examples and special solutions}%
In this section, we only consider time-independent coupling
constants $\alpha(t)\equiv \alpha$. First, we study two very simple
homogeneous examples for the $(RH)_\alpha$ flow system to illustrate
the different behavior of the flow for different coupling constants
$\alpha$. In particular, the existence or non-existence of
singularities will depend on the choice of $\alpha$. We study the
volume-preserving version of the flow as well. We say that
$(g(t),\phi(t))$ is a solution of the normalized $(RH)_\alpha$ flow,
if it satisfies
\begin{equation}\label{4.eq1}
\left\{\begin{aligned}\dt g &= -2\Rc{} + 2\alpha \D\phi \otimes
\D\phi + \tfrac{2}{m}\,g\;\dashint_M \big(R-\alpha\abs{\D\phi}^2\big)dV,\\%
\dt \phi &= \tau_g \phi. \end{aligned}\right.
\end{equation}

\subsection{Two homogeneous examples with $\phi=\id$}
Assume that $(M,g(0))$ is a round two-sphere of constant Gauss
curvature $1$. Under the Ricci flow, the sphere shrinks to a point
in finite time. Let us now consider the $(RH)_\alpha$ flow, assuming
that $(N,\gamma)=(M,g(0))$ and $\phi(0)$ is the identity map. With the ansatz $g(t)=c(t)g(0)$, $c(0)=1$ and the fact that $\phi(t)=\phi(0)$ is
harmonic for all $g(t)$, the $(RH)_\alpha$ flow reduces to
\begin{equation*}
\dt c(t) = -2+2\alpha.
\end{equation*}
For $\alpha < 1$, $c(t)$ goes to zero in finite time, i.e.~$(M,g(t))$ 
shrinks to a point, while the scalar curvature $R$
and the energy density $\abs{\D\phi}^2$ both go to infinity. For
$\alpha=1$, the solution is stationary. For $\alpha>1$, $c(t)$ grows
linearly and the flow exists forever, while both the scalar
curvature $R$ and the energy density $\abs{\D\phi}^2$ vanish
asymptotically. Instead of changing $\alpha$, we can also scale the
metric $\gamma$ on the target manifold. Mapping into a larger sphere
has the same consequences as choosing a larger $\alpha$. 
Note that the volume-preserving version (\ref{4.eq1}) 
of the flow is always stationary, as it is for the normalized Ricci flow, too.
\newline

A more interesting example is obtained if we let
$(M^4,g(t))=(\mathbb{S}^2\times L,c(t)g_{\mathbb{S}^2}\oplus
d(t)g_L)$, where $(\mathbb{S}^2,g_{\mathbb{S}^2})$ is again a round
sphere with Gauss curvature $1$ and $(L,g_L)$ is a surface (of
genus $\geq 2$) with constant Gauss curvature $-1$. Under the Ricci
flow, $\dt c(t)=-2$ and $\dt d(t)=+2$. In particular, $c(t)$ goes to
zero in finite time while $d(t)$ always expands. Under the
normalized Ricci flow $\dt g=-2\Rc{}+\tfrac{1}{2}\,g\;\dashint R\;
dV$, we have
\begin{equation*}
\dt c = -2 + \tfrac{d-c}{d}=-1-c^2,\qquad \dt
d=+2+\tfrac{d-c}{c}=+1+d^2.
\end{equation*}
Again, $c(t)$ goes to zero in finite time. At the same time, $d(t)$
goes to infinity. Now, let us consider the $(RH)_\alpha$ flow for this example,
setting $(N,\gamma)=(M,g(0))$ and $\phi(0)=\id$. First,
note that $\phi(0)$ is always harmonic and thus $\phi(t)=\phi(0)$ is
unchanged. The identity map between the same manifold with
two different metrics is not necessarily harmonic in general, but here
it is. The flow equations reduce to $\dt c(t)=-2+2\alpha$
and $\dt d(t)=+2+2\alpha$. While $d(t)$ always grows, the behavior
of $c(t)$ is exactly the same as in the first example above, where
we only had a two-sphere. On the other hand, if we consider the
normalized flow (\ref{4.eq1}), we obtain 
\begin{equation*}
\dt c = (\alpha-1) -(\alpha+1)c^2,\qquad \dt d=(\alpha+1) -
(\alpha-1)d^2.
\end{equation*}
In the case where $\alpha <1$, $c(t)$ goes to zero in finite time,
while $d(t)$ blows up at the same time, similar to the normalized
Ricci flow above. For $\alpha=1$, $c(t)=(1+2t)^{-1}$ goes to zero in
infinite time while $d(t)=(1+2t)$ grows linearly, i.e.~we have
long-time existence but no natural convergence (to a manifold with
the same topology). In the third case, where $\alpha
>1$, both $c(t)$ and $d(t)$ converge with
\begin{equation*}
c(t) \to \sqrt{\tfrac{\alpha-1}{\alpha+1}}, \qquad d(t) \to
\sqrt{\tfrac{\alpha+1}{\alpha-1}}, \qquad \textrm{as }t\to\infty.
\end{equation*}
These examples show that both the unnormalized and the normalized 
version of our flow can behave very differently from the Ricci flow if $\alpha$
is chosen large. In particular, they may be more regular in special situations.

\subsection{Volume-preserving version of $(RH)_\alpha$}
Here, we show that the unnormalized $(RH)_\alpha$ flow and the
normalized version are related by rescaling the metric $g$ and the
time, while keeping the map $\phi$ unchanged. Indeed, assume that
$(g(t),\phi(t))_{t\in[0,T)}$ is a solution of $(RH)_\alpha$. Define
a family of rescaling factors $\lambda(t)$ by
\begin{equation}\label{4.eq9}
\lambda(t) := \bigg(\int_M dV_{g(t)}\bigg)^{\!-2/m}, \quad t\in
[0,T),
\end{equation}
and let $\bar{g}(t)$ be the family of rescaled metrics
$\bar{g}(t)=\lambda(t)g(t)$ having constant unit volume
\begin{equation*}
\int_M dV_{\bar{g}(t)} = \int_M \lambda^{m/2}(t)\; dV_{g(t)} =1,
\quad \forall t\in [0,T).
\end{equation*}
Write $\sS=\Rc{}-\alpha\D\phi\otimes\D\phi$ with trace
$S=R-\alpha\abs{\D\phi}^2$. It is an immediate consequence of the
scaling behavior of $\Rc{}$, $R$, $\D\phi\otimes\D\phi$ and
$\abs{\D\phi}^2$ that $\bar{\sS}=\sS$ and $\bar{S}=\lambda^{-1}S$
for $\bar{g}=\lambda g$. Note that $\lambda(t)$ is a smooth function
of time, with
\begin{equation*}
\frac{d}{dt} \lambda(t) = -\tfrac{2}{m}\bigg(\int_M
dV_g\bigg)^{\!-\frac{2+m}{m}}\!\int_M (-S) dV_g=
\tfrac{2}{m}\,\lambda\; \dashint_M S\; dV_g =
\tfrac{2}{m}\,\lambda^2\; \dashint_M \bar{S}\; dV_{\bar{g}}.
\end{equation*}
Now, we rescale the time. Put $\bar{t}(t):=\int_0^t\lambda(s) ds$,
so that $\tfrac{d\bar{t}}{dt}=\lambda(t)$. Then, we obtain
\begin{align*}
\tfrac{\partial}{\partial \bar{t}}\bar{g} &= \lambda^{-1}\dt(\lambda
g ) = \dt g + \big(\lambda^{-2}\dt\lambda\big)\bar{g} = -2\bar{\sS}
+\tfrac{2}{m}\,\bar{g}\;\dashint_M \bar{S}\; dV_{\bar{g}},\\%
\tfrac{\partial}{\partial \bar{t}}\phi &= \lambda^{-1}\dt \phi =
\lambda^{-1}\tau_g\phi = \tau_{\bar{g}}\phi.
\end{align*}
This means that $(\bar{g}(\bar{t}),\phi(\bar{t}))$ solves the
volume-preserving $(RH)_\alpha$ flow (\ref{4.eq1}) on $[0,\bar{T})$,
where $\bar{T}=\int_0^T \lambda(s) ds$.

\subsection{Gradient solitons}
A solution to $(RH)_\alpha$ which changes under a one-parameter
family of diffeomorphisms on $M$ and scaling is called a soliton (or
a self-similar solution). These solutions correspond to fixed points
modulo diffeomorphisms and scaling. The more general class of
periodic solutions modulo diffeomorphisms and scaling, the so-called
breathers, will be defined (but also ruled out) in Section 7.
\begin{defn}\label{4.def1}
A solution $(g(t),\phi(t))_{t\in[0,T)}$ of\/ $(RH)_\alpha$ is called
a \emph{soliton} if there exists a one-parameter family of
diffeomorphisms $\psi_t:M\to M$ with $\psi_0=\id_M$ and a scaling
function $c:[0,T)\to\RR_{+}$ such that
\begin{equation}\label{4.eq2}
\left\{\begin{aligned}g(t)&=c(t)\psi_t^*g(0),\\%
\phi(t)&=\psi_t^*\phi(0).\end{aligned}\right.
\end{equation}
The cases $\dt c=\dot{c}<0$, $\dot{c}=0$ and $\dot{c}>0$ correspond
to \emph{shrinking}, \emph{steady} and \emph{expanding} solitons,
respectively. If the diffeomorphisms $\psi_t$ are generated by a
vector field $X(t)$ that is the gradient
of some function $f(t)$ on $M$, then the soliton is called
\emph{gradient} soliton and $f$ is called the \emph{potential} of
the soliton.
\end{defn}
\begin{lemma}\label{4.lemma2}
Let $(g(t),\phi(t))_{t\in[0,T)}$ be a gradient soliton with
potential $f$. Then for any $t_0\in[0,T)$, this soliton satisfies
the coupled elliptic system
\begin{equation}\label{4.eq3}
\left\{\begin{aligned}0 &=\Rc{}-\alpha\D\phi\otimes\D\phi+\Hess(f)+\sigma g,\\%
0 &= \tau_g\phi-\scal{\D\phi,\D f},\end{aligned}\right.
\end{equation}
for some constant $\sigma(t_0)$. Conversely, given a function $f$ on 
$M$ and a solution of (\ref{4.eq3}) at $t=0$, there exist one-parameter families 
of constants $c(t)$ and diffeomorphisms $\psi_t:M\to M$ such that defining 
$(g(t),\phi(t))$ as in (\ref{4.eq2}) yields a solution of\/ $(RH)_\alpha$. 
Moreover, $c(t)$ can be chosen linear in $t$.
\end{lemma}
\begin{proof}
Suppose we have a soliton solution to $(RH)_\alpha$. Without loss of
generality, $c(0)=1$ and $\psi_0=\id_M$. Thus, the solution
satisfies
\begin{align*}
-2\Rc{g(0)}+2\alpha(\D\phi\otimes\D\phi)(0) &= \dt g(t)|_{t=0} =
\dt(c(t)\psi_t^*g(0))|_{t=0}\\%
&= \dot{c}(0)g(0)+\sL_{X(0)}g(0) =
\dot{c}(0)g(0)+2\Hess(f(\cdot,0)),
\end{align*}
where $X(t)$ is the family of vector fields generating $\psi_t$.
Moreover, we compute
\begin{equation*}
(\tau_g\phi)(0) = \dt \phi(t)|_{t=0} = \sL_{X(0)}\phi(0) =
\scal{\D\phi,\D f}.
\end{equation*}
Together, this proves (\ref{4.eq3}) with
$\sigma=\tfrac{1}{2}\dot{c}(0)$ for $t_0=0$. Hence, by a
time-shifting argument, (\ref{4.eq3}) must hold for any
$t_0\in[0,T)$. One can easily see that
$\sigma(t_0)=\dot{c}(t_0)/2c(t_0)$.%
\newline

Conversely, let $(g(0),\phi(0))$ solve (\ref{4.eq3}) for some
function $f$ on $M$. Define $c(t):=1+2\sigma t$ and $X(t):=\D
f/c(t)$. Let $\psi_t$ be the diffeomorphisms generated by the family
of vector fields $X(t)$ (with $\psi_0=\id_M$) and define
$(g(t),\phi(t))$ as in (\ref{4.eq2}). For $\sigma<0$ this is
possible on the time interval $t\in[0,\tfrac{-1}{2\sigma})$, in the
case $\sigma\geq 0$ it is possible for $t\in[0,\infty)$. Then
\begin{align*}
\dt g(t) &= \dot{c}(t)\psi_t^*(g(0)) + c(t)\psi_t^*(\sL_{X(t)}g(0))
= \psi_t^*(2\sigma g(0) +\sL_{\D f}g(0))\\%
&= \psi_t^*(2\sigma g(0) + 2\Hess(f)) =
\psi_t^*\big({-2}\Rc{g(0)}+2\alpha(\D\phi\otimes\D\phi)(0)\big)\\%
&= -2\Rc{g(t)}+2\alpha(\D\phi\otimes\D\phi)(t),
\end{align*}
as well as
\begin{align*}
\dt\phi(t)&=\psi_t^*(\sL_{X(t)}\phi(0))=\psi_t^*\scal{\D\phi(0),\D
f/c(t)} = c(t)^{-1}\psi_t^*(\tau_{g(0)}\phi(0))\\%
&= c(t)^{-1}\tau_{\psi_t^*g(0)}\phi(t) = \tau_{g(t)}\phi(t).
\end{align*}
This means that $(g(t),\phi(t))$ is a solution of $(RH)_\alpha$ and
thus a soliton solution.
\end{proof}
By Lemma \ref{4.lemma2} and rescaling, we may assume that $c(t)=T-t$
for shrinking solitons (here $T$ is the maximal time of existence
for the flow), $c(t)=1$ for steady solitons and $c(t)=t-T$ for
expanding solitons (where $T$ defines a \emph{birth} time). An
example for a soliton solution is the very first example from this
section, where $(M,g(0))=(N,\gamma)=(\mathbb{S}^2,g_{\mathbb{S}^2})$
and $\phi(0)=\id$. For $\alpha<1$, the soliton is shrinking, for
$\alpha=1$ steady and for $\alpha>1$ expanding. Since $\phi_t=\id_M$
for all $t$ in all three cases, these are gradient solitons with
potential $f=0$.%
\newline

Taking the trace of the first equation in (\ref{4.eq3}), we see 
that a soliton must satisfy
\begin{equation}\label{4.eq4}
R-\alpha\abs{\D\phi}^2+\Lap f + \sigma m = 0.
\end{equation}
Taking covariant derivatives in (\ref{4.eq3}) and using the 
twice traced second Bianchi identity $\D_jR_{ij}=\tfrac{1}{2}\D_iR$, we obtain 
(analogous to the corresponding equation for soliton solutions of
the Ricci flow)
\begin{equation}\label{4.eq5}
R-\alpha\abs{\D\phi}^2+\abs{\D f}^2 + 2\sigma f = const.
\end{equation}
Finally, with $f(\cdot,t)=\psi_t^*(f(\cdot,0))$, we get
\begin{equation}\label{4.eq6}
\dt f = \sL_Xf = \abs{\D f}^2.
\end{equation}
Combining this with (\ref{4.eq4}), we obtain the evolution equation
\begin{equation}\label{4.eq7}
(\dt + \Lap)f = \abs{\D f}^2-R+\alpha\abs{\D\phi}^2-\sigma m.
\end{equation}
For steady solitons, for which the formulas
(\ref{4.eq3})--(\ref{4.eq7}) hold with $\sigma=0$, equation (\ref{4.eq7})
is equivalent to $u= e^{-f}$ solving the adjoint heat equation
\begin{equation}\label{4.eq8}
\Box^*u=-\dt u-\Lap u +Ru-\alpha \abs{\D \phi}^2 u = 0.
\end{equation}
For shrinking solitons, (\ref{4.eq3})--(\ref{4.eq7}) hold with
$\sigma(t)=-\tfrac{1}{2}(T-t)^{-1}$ and (\ref{4.eq7}) is equivalent
to $u= (4\pi(T-t))^{-m/2}e^{-f}$ solving the adjoint heat equation.
Finally, for expanding solitons, $\sigma(t)=+\tfrac{1}{2}(t-T)^{-1}$
and (\ref{4.eq7}) is equivalent to the fact that $u=
(4\pi(t-T))^{-m/2}e^{-f}$ solves the adjoint heat equation
(\ref{4.eq8}).

\section{The $(RH)_\alpha$ flow as a gradient flow}
In this section, we introduce an energy functional $\sF_\alpha$ for
the $(RH)_\alpha$ flow, which corresponds to Perelman's $\sF$-energy
for the Ricci flow introduced in \cite[Section 1]{Perelman:entropy}.
For a detailed study of Perelman's functional, we refer to Chow et
al. \cite[Chapter 5]{RF:TAI}, M\"{u}ller \cite[Chapter
3]{Muller:Harnack}, or Topping \cite[Chapter 6]{Topping:Lectures}.
In the special case $N\subseteq\RR$, the corresponding functional was 
introduced by List \cite{List:diss}. We follow his work closely in the first 
part of this section.

\subsection{The energy functional and its first variation}%
Let $g=g_{ij}\in \Gamma\big(\Sym^2_{+}(T^*M)\big)$ be a Riemannian
metric on a closed manifold $M$, $f:M\to\RR$ a smooth function and
$\phi\in C^\infty(M,N):=\{\phi\in C^\infty(M,\RR^d)\mid \phi(M)
\subseteq N\}$. For a constant $\alpha(t)\equiv \alpha>0$, we set
\begin{equation}\label{1.eq1}
\sF_\alpha(g,\phi,f) := \int_M \Big(R_g + \abs{\D f}^2_g -
\alpha\abs{\D\phi}^2_g\Big)e^{-f}dV_g.
\end{equation}
Take variations
\begin{align*}
g_{ij}^\eps &= g_{ij}+\eps h_{ij}, &h_{ij} &\in\Gamma
\big(\Sym^2(T^*M)\big),\\%
f^\eps &= f+\eps \ell, &\ell &\in C^\infty(M),\\%
\phi^\eps &= \pi_N(\phi + \eps \vartheta), &\vartheta &\in
C^\infty(M,\RR^d) \text{ with } \vartheta(x) \in T_{\phi(x)}N,
\end{align*}
where $\pi_N$ is the smooth nearest-neighbour projection defined on
a tubular neighbourhood of $N \subset \RR^d$. Note that we used the
identification $T_pN \subset T_p\RR^d \cong \RR^d$. We denote by
$\delta$ the derivative $\delta =\frac{d}{d\eps}\big|_{\eps=0}$,
i.e.~we have $\delta g=h$, $\delta f = \ell$ and $\delta \phi =
(d\pi_N \circ \phi)\vartheta = \vartheta$. Our goal is to compute
\begin{align*}
\delta \sF_{\alpha,g,\phi,f}(h,\vartheta,\ell)&:=
\frac{d}{d\eps}\Big|_{\eps=0} \sF_\alpha(g+\eps h,\pi_N(\phi +
\eps\vartheta), f+\eps\ell)\\%
&\phantom{:}= \underbrace{\delta\int_M \big(R + \abs{\D f}^2\big)
e^{-f}dV}_{=:\;I} - \;\alpha \cdot \underbrace{\delta\int_M
\abs{\D\phi}^2 e^{-f}dV}_{=:\;I\!I}.
\end{align*}
For the variation of the first integral, we know from Ricci flow
theory that
\begin{equation*}
I = \int_M \Big({-h^{ij}}\big(R_{ij} + \D_i\D_j f\big) +
\big(\tfrac{1}{2} \tr_g h - \ell\big)\big(2\Lap f - \abs{\D f}^2
+R\big)\Big) e^{-f}dV,
\end{equation*}
see Perelman \cite[Section 1]{Perelman:entropy}, or M\"{u}ller
\cite[Lemma 3.3]{Muller:Harnack}, for a detailed proof. For the
variation of the second integral, we compute with a partial
integration
\begin{align*}
I\!I &= \int_M 2g^{ij}\D_i \phi^\lambda \D_j \vartheta^\lambda e^{-f}dV +
\int_M \Big(-h^{ij}\D_i\phi^\lambda \D_j \phi^\lambda +
\abs{\D\phi}^2_g \big(\tfrac{1}{2}\tr_g h - \ell\big)\Big)
e^{-f}dV,\\%
&= \int_M \Big({-2}\vartheta^\lambda\big(\Lap_g \phi^\lambda
-\Scal{\D\phi^\lambda,\D f}_g\big) -h^{ij}\D_i\phi^\lambda \D_j
\phi^\lambda + \abs{\D\phi}^2_g \big(\tfrac{1}{2}\tr_g h -
\ell\big)\Big) e^{-f}dV.
\end{align*}
Hence, by putting everything together, we find
\begin{align}
\delta \sF_{\alpha,g,\phi,f}(h,\vartheta,\ell) &= \int_M
-h^{ij}\big(R_{ij} + \D_i\D_j f-\alpha \D_i \phi \D_j
\phi\big)e^{-f}dV\notag\\%
&\quad+ \int_M \big(\tfrac{1}{2}\tr_g h - \ell\big)\big(2\Lap f -
\abs{\D f}^2 + R - \alpha \abs{\D \phi}^2\big)e^{-f}dV\label{1.eq2}\\%
&\quad+\int_M 2\alpha\vartheta \big(\tau_g \phi - \Scal{\D\phi,\D
f}\big) e^{-f}dV,\notag
\end{align}
since $\vartheta\Lap_g \phi = \vartheta\tau_g \phi$, where $\tau_g
\phi := \Lap_g \phi - A(\phi)(\D\phi,\D\phi)_M$ denotes the tension
field of $\phi$.

\subsection{Gradient flow for fixed background measure}%
Now, we fix the measure $d\mu = e^{-f}dV$, i.e.~let $f=
-\log\big(\frac{d\mu}{dV}\big)$, where $\frac{d\mu}{dV}$ denotes the
Radon-Nikodym differential of measures. Then, from $0= \delta d\mu =
\left(\frac{1}{2}\tr_g h -\ell\right)d\mu$ we get $\ell =
\frac{1}{2}\tr_g h$. Thus, for a fixed measure $\mu$ the functional
$\sF_\alpha$ and its variation $\delta \sF_\alpha$ depend only on
$g$ and~$\phi$ and their variations $\delta g = h$ and $\delta \phi
= \vartheta$. In the following we write
\begin{equation}\label{1.eq3}
\sF^\mu_\alpha(g,\phi) :=
\sF_\alpha\big(g,\phi,-\log\big(\tfrac{d\mu}{dV}\big)\big)
\end{equation}
and
\begin{equation*}
\delta \sF^\mu_{\alpha,g,\phi}(h,\vartheta) := \delta
\sF_{\alpha,g,\phi,-\log(\frac{d\mu}{dV})}
\left(h,\vartheta,\tfrac{1}{2}\tr_g h\right).
\end{equation*}
Equation (\ref{1.eq2}) reduces to
\begin{equation}\label{1.eq4}
\delta \sF^\mu_{\alpha,g,\phi}(h,\vartheta) = \int_M
\Big({-h^{ij}}\big(R_{ij}+\D_i\D_j f - \alpha \D_i \phi
\D_j\phi\big)+ 2\alpha\vartheta \big(\tau_g \phi-\Scal{\D\phi,\D
f}\big) d\mu.
\end{equation}
Let $(g,\phi)\in \Gamma(\Sym^2_{+}(T^*M))\times C^{\infty}(M,N)$ and
define on $H:=H_{g,\phi}=\Gamma(\Sym^2(T^*M))\times T_\phi
C^\infty(M,N)$ an inner product depending on $\alpha$ and the
measure $\mu$ by
\begin{equation*}
\Scal{(k_{ij},\psi),(h_{ij},\vartheta)}_{H,\alpha,\mu}:= \int_M
\big(\tfrac{1}{2} h^{ij}k_{ij} + 2\alpha\psi\vartheta\big)d\mu.
\end{equation*}
From $\delta \sF^\mu_{\alpha,g,\phi}(h,\vartheta) = \Scal{\grad
\sF^\mu_\alpha(g,\phi),(h,\vartheta)}_{H,\alpha,\mu}$ we then deduce
\begin{equation}\label{1.eq5}
\grad \sF^\mu_\alpha(g,\phi) = \big(-2(R_{ij}+\D_i\D_j f -\alpha\D_i
\phi \D_j \phi),\; \tau_g \phi -\Scal{\D\phi,\D f}\big).
\end{equation}
Let $\pi_1$, $\pi_2$ denote the natural projections of $H$ onto its
first and second factors, respectively. Then, the gradient flow of
$\sF^\mu_\alpha$ is
\begin{equation*}
\left\{\begin{aligned}\dt g_{ij} &= \pi_1(\grad
\sF^\mu_\alpha(g,\phi)),\\%
\dt\phi &=\pi_2(\grad \sF^\mu_\alpha(g,\phi)).\end{aligned}\right.
\end{equation*}
Thus, recalling the equation $\dt f = \ell = \frac{1}{2}\tr_g
\!\big(\dt g_{ij}\big)$, we obtain the gradient flow system
\begin{equation}\label{1.eq6}
\left\{\begin{aligned}\dt g_{ij} &= -2(R_{ij}+\D_i\D_j f -\alpha
\D_i
\phi \D_j \phi),\\%
\dt \phi &= \tau_g \phi - \Scal{\D\phi,\D f},\\%
\dt f &= -R -\Lap f + \alpha \abs{\D\phi}^2.\end{aligned}\right.
\end{equation}%

\subsection{Pulling back with diffeomorphisms}%
As one can do for the Ricci flow (see Perelman \cite[Section
1]{Perelman:entropy} or M\"{u}ller \cite[page 52]{Muller:Harnack}), we
now pull back a solution $(g,\phi,f)$ of (\ref{1.eq6}) with a family
of diffeomorphisms generated by $X=\D f$. Indeed, recalling the
formulas for the Lie derivatives $(\mathcal{L}_{\D f}g)_{ij}= 2
\D_i\D_j f$, $\mathcal{L}_{\D f}\phi = \Scal{\D\phi,\D f}$ and
$\mathcal{L}_{\D f}f = \abs{\D f}^2$, we can rewrite (\ref{1.eq6})
in the form
\begin{equation*}
\left\{\begin{aligned}\dt g &= -2\Rc{}+2\alpha\D \phi \otimes\D
\phi - \big(\mathcal{L}_{\D f}g\big),\\%
\dt \phi &= \tau_g \phi - \big(\mathcal{L}_{\D f}\phi\big),\\%
\dt f &= -\Lap f + \abs{\D f}^2 -R + \alpha \abs{\D\phi}^2 -
\big(\mathcal{L}_{\D f}f\big).\end{aligned}\right.
\end{equation*}
Hence, if $\psi_t$ is the one-parameter family of diffeomorphisms
induced by the vector field $X(t)=\D f(t) \in \Gamma(TM)$, $t \in
[0,T)$, i.e.~if $\dt \psi_t = X(t) \circ \psi_t$, $\psi_0 = \id$,
then the pulled-back quantities $\tilde{g}=\psi^*_t g$,
$\tilde{\phi}=\psi^*_t \phi$, $\tilde{f}=\psi^*_t f$ satisfy
\begin{equation*}
\left\{\begin{aligned}\dt \tilde{g} &= -2\tilde{\Rc{}}+2\alpha
\D \tilde{\phi} \otimes \D \tilde{\phi},\\%
\dt \tilde{\phi} &= \tau_{\tilde{g}} \tilde{\phi},\\%
\dt \tilde{f} &= -\Lap_{\tilde{g}} \tilde{f} + \abs{\D
\tilde{f}}^2_{\tilde{g}} -\tilde{R} + \alpha
\abs{\D\tilde{\phi}}^2_{\tilde{g}}.\end{aligned}\right.
\end{equation*}
Here, $\tilde{\Rc{}}$ and $\tilde{R}$ denote the Ricci and scalar
curvature of $\tilde{g}$ and $\Lap$, $\tau$ and the norms are also
computed with respect to $\tilde{g}$. In the following, we will
usually consider the pulled-back gradient flow system and therefore
drop the tildes for convenience of notation.%
\newline

Note that the formal adjoint of the heat operator $\Box=\dt-\Lap$
under the flow $\dt g = h$ is $\Box^*=-\dt-\Lap-\frac{1}{2}\tr_g h$.
Indeed, for functions $v,w\colon M\times[0,T]\to \RR$, a
straightforward computation yields
\begin{equation*}
\int_0^T \int_M (\Box v)w \; dV\,dt = \bigg[\int_M vw \; dV
\bigg]_0^T + \int_0^T \int_M v(\Box^*w) \; dV \,dt.
\end{equation*}
In our case where $h_{ij}=-2R_{ij}+2\alpha \D_i\phi \D_j\phi$, this
is $\Box^*=-\dt-\Lap +R-\alpha \abs{\D \phi}^2$ and thus the
evolution equation for $f$ is equivalent to $e^{-f}$ solving the
adjoint heat equation $\Box^*e^{-f}=0$. The system now reads
\begin{equation}\label{1.eq7}
\left\{\begin{aligned}\dt g &= -2\Rc{} + 2\alpha \D\phi \otimes \D\phi,\\%
\dt \phi &= \tau_g \phi,\\%
0 &= \Box^*e^{-f}. \end{aligned}\right.
\end{equation}
This means that $(RH)_\alpha$ can be interpreted as the
gradient flow of $\sF^\mu_\alpha$ for any fixed background measure
$\mu$. Moreover, using (\ref{1.eq4}), (\ref{1.eq5}) and the
diffeomorphism invariance of $\sF_\alpha$, we get the following.
\begin{prop}\label{1.prop1}
Let $(g(t),\phi(t))_{t \in [0,T)}$ be a solution of the
$(RH)_\alpha$ flow with coupling constant $\alpha(t)\equiv \alpha>0$
and let\/ $e^{-f}$ solve the adjoint heat equation under this flow.
Then the energy functional $\sF_\alpha(g,\phi,f)$ defined in
(\ref{1.eq1}) is non-decreasing with
\begin{equation}\label{1.eq8}
\frac{d}{dt} \sF_\alpha = \int_M \Big(2\Abs{\Rc{}-\alpha \D\phi
\otimes \D\phi+\Hess(f)}^2 +2\alpha \Abs{\tau_g \phi-\Scal{\D\phi,\D
f}}^2\Big) e^{-f}dV \geq 0.
\end{equation}
Moreover, $\sF_\alpha$ is constant if and only if $(g(t),\phi(t))$
is a steady soliton.
\end{prop}
Allowing also time-dependent coupling constants $\alpha(t)$, we
obtain the following.
\begin{cor}\label{1.cor2}
Let $(g(t),\phi(t))_{t \in [0,T)}$ solve $(RH)_\alpha$ for a
positive coupling function $\alpha(t)$ and let\/ $e^{-f}$ solve the
adjoint heat equation under this flow. Then
$\sF_{\alpha(t)}(g(t),\phi(t),f(t))$ satisfies
\begin{equation*}
\frac{d}{dt} \sF_\alpha = \int_M \Big(2\Abs{\Rc{}-\alpha \D\phi
\otimes \D\phi+\Hess(f)}^2 +2\alpha \Abs{\tau_g \phi-\Scal{\D\phi,\D
f}}^2 - \dot{\alpha} \abs{\D\phi}^2 \Big) e^{-f}dV,
\end{equation*}
in particular, it is non-decreasing if $\alpha(t)$ is a
non-increasing function.
\end{cor}

\subsection{Minimizing over all probability measures}%
Following Perelman \cite{Perelman:entropy}, we define
\begin{equation}\label{1.eq13}
\lambda_\alpha(g,\phi):= \inf\big\{\sF^\mu_\alpha(g,\phi)\;\big|\;
\mu(M)=1\big\} = \inf\left\{\sF_\alpha(g,\phi,f)\;\bigg|\; \int_M
e^{-f}dV =1\right\}.
\end{equation}
The first task is to show that the infimum is always achieved.
Indeed, if we set $v=e^{-f/2}$, we can write the energy as
\begin{equation*}
\sF_\alpha(g,\phi,v)=\int_M \Big(Rv^2 + 4\abs{\D
v}^2-\alpha\abs{\D\phi}^2v^2\Big)dV = \int_M v\Big(Rv - 4\Lap v -
\alpha\abs{\D\phi}^2v\Big)dV.
\end{equation*}
Hence
\begin{equation*}
\lambda_\alpha(g,\phi)=\inf\left\{\int_M v\Big(Rv - 4\Lap v -
\alpha\abs{\D\phi}^2v\Big)dV \;\bigg| \; \int_M v^2 dV=1\right\}
\end{equation*}
is the smallest eigenvalue of the operator $-4\Lap{} +R
-\alpha\abs{\D\phi}^2$ and $v$ is a corresponding normalized
eigenvector. Since the operator (for any time $t$ and map $\phi(t)$)
is a Schr\"{o}dinger operator, there exists a unique positive and
normalized eigenvector $v_{min}(t)$, see for example Reed and Simon
\cite{ReedSimon} and Rothaus \cite{Rothaus:Schroedinger}. From
eigenvalue perturbation theory, we see that if $g(t)$ and $\phi(t)$
depend smoothly on $t$, then so do $\lambda_\alpha(g(t),\phi(t))$
and $v_{min}(t)$.
\begin{prop}\label{1.prop3}
Let $(g(t),\phi(t))_{t\in[0,T)}$ be a smooth solution of the
$(RH)_\alpha$ flow with constant $\alpha(t)\equiv\alpha>0$. Then
$\lambda_\alpha(g,\phi)$ as defined in (\ref{1.eq13}) is monotone
non-decreasing in time and it is constant if and only if
\begin{equation}\label{1.eq14}
\left\{\begin{aligned}0 &= \Rc{}-\alpha\D\phi\otimes\D\phi+\Hess(f),\\%
0 &= \tau_g \phi -\scal{\D\phi,\D f}, \end{aligned}\right.
\end{equation}
for the minimizing function $f=-2\log v_{min}$.
\end{prop}
\begin{proof}
Pick arbitrary times $t_1$, $t_2 \in [0,T)$ and let $v_{min}(t_2)$
be the unique positive minimizer for $\lambda_\alpha
(g(t_2),\phi(t_2))$. Put $u(t_2)=v_{min}^2(t_2)>0$ and solve the
adjoint heat equation $\Box^*u=0$ backwards on $[t_1,t_2]$. Note
that $u(x',t')>0$ for all $x'\in M$ and $t'\in [t_1,t_2]$ by the
maximum principle and the constraint $\int_M u \;dV=\int_M v^2 \;dV
=1$ is preserved since
\begin{equation*}
\frac{d}{dt} \int_M u\;dV = \int_M \big(\dt u\big)dV + \int_M
u\big(\dt dV\big) = -\int_M\Lap u\; dV = 0.
\end{equation*}
Here, we used $\dt dV = -\tfrac{1}{2}\tr_g(\dt g)dV =
(-R+\alpha\abs{\D\phi}^2)dV$ and $\dt u = (-\Lap +R-\alpha \abs{\D
\phi}^2)u$, the latter following from $\Box^*u=0$. Thus, with $u(t)=
e^{-\bar{f}(t)}$ for all $t\in [t_1,t_2]$, we obtain with
Proposition \ref{1.prop1}
\begin{equation*}
\lambda_\alpha(g(t_1),\phi(t_1))\leq
\sF_\alpha(g(t_1),\phi(t_1),\bar{f}(t_1)) \leq
\sF_\alpha(g(t_2),\phi(t_2),\bar{f}(t_2)) =
\lambda_\alpha(g(t_2),\phi(t_2)).
\end{equation*}
The condition (\ref{1.eq14}) in the equality case follows directly
from (\ref{1.eq8}).
\end{proof}
Again the monotonicity of $\lambda_\alpha(g,\phi)$ is preserved if
we allow a positive non-increasing coupling function $\alpha(t)$ 
instead of a time-independent positive constant $\alpha$.

\section{Short-time existence and evolution equations}
Due to diffeomorphism invariance, our flow is only weakly parabolic.
In fact, the principal symbol for the first equation is the same as
for the Ricci flow since the additional term is of lower order.
Thus, one cannot directly apply the standard parabolic existence
theory. Fortunately, shortly after Hamilton's first proof of
short-time existence for the Ricci flow in \cite{Hamilton:3folds}
which was based on the Nash-Moser implicit function theorem, DeTurck
\cite{DeTurck:trick} found a substantially simpler proof which can
easily be modified to get an existence proof for our system
$(RH)_\alpha$. Note that we only consider the case where $M$ is
closed, but following Shi's short-time existence proof in
\cite{Shi:complete} for the Ricci flow on complete noncompact
manifolds, one can also prove more general short-time existence
results for our flow. We first recall some results for the Ricci
flow, following the presentation of Hamilton \cite{Hamilton:survey}
very closely.%
\newline

Since some results strongly depend on the curvature of $(N,\gamma)$,
it is more convenient to work with $\phi:M\to N$ itself instead of
$e_N\circ\phi:M\to \RR^d$ as in the last section. Therefore,
repeated Greek indices are summed over from $1$ to $n = \dim N$ in
this section.

\subsection{Dual Ricci-Harmonic and Ricci-DeTurck flow}%
Let $g(t)$ be a solution of the Ricci flow and
$\psi(t)\colon(M,g)\to (M,h)$ a one parameter family of smooth maps
satisfying the harmonic map flow $\dt \psi = \tau_g
\psi$ with respect to the evolving metric $g$. Note that this is $(RH)_{\alpha\equiv 0}$. 
If $\psi(t)$ is a diffeomorphism at time
$t=0$, it will stay a diffeomorphism for at least a short time. Now,
we consider the push-forward $\tilde{g}:=\psi_*g$ of the metric $g$
under $\psi$. The evolution equation for $\tilde{g}$ reads
\begin{equation}\label{2.eq1}
\dt \tilde{g}_{ij}= -2\tilde{R}_{ij} + (\mathcal{L}_V
\tilde{g})_{ij} = -2\tilde{R}_{ij} + \tilde{\D}_iV_j +
\tilde{\D}_jV_i,
\end{equation}
where $\tilde{\D}$ denotes the Levi-Civita connection of $\tilde{g}$
and $\dt \psi = -V \circ \psi$. One calls this the dual
Ricci-Harmonic flow or also the $h$-flow. An easy computation (see
DeTurck \cite{DeTurck:trick} or Chow and Knopf \cite[Chapter
3]{RF:intro}) shows that $V$ is given by
\begin{equation}\label{2.eq2}
V^\ell=\tilde{g}^{ij}\big(\!\phantom{l}^{\tilde{g}}\Gamma^\ell_{ij}-
\!\phantom{l}^h\Gamma^\ell_{ij}\big),
\end{equation}
the trace of the tensor which is the difference between the
Christoffel symbols of the connections of $\tilde{g}$
and of $h$, respectively. Note that the evolution equation of
$\tilde{g}$ involves only the metrics $\tilde{g}$ and $h$ and not
the metric $g$, and since it involves $\!\phantom{l}^h\Gamma$ for
the fixed background metric $h$ it is no longer diffeomorphism
invariant. Indeed, one can show (see Hamilton \cite[Section
6]{Hamilton:survey}) that
\begin{equation}\label{2.eq3}
\dt \tilde{g}_{ij} = \tilde{g}^{k\ell}\tilde{\D}_k \hat{\D}_\ell
\tilde{g}_{ij},
\end{equation}
where $\hat{\D}$ denotes the connection of the
background metric $h$. Since $\hat{\D}$ is independent of
$\tilde{g}$ and $\tilde{\D}$ only involves first derivatives of
$\tilde{g}$ this is a quasilinear equation. Its principal symbol is
$\sigma(\xi)=\tilde{g}^{ij}\xi_i \xi_j\cdot \id$, where $\id$ is the
identity on tensors $\tilde{g}$. Hence this flow equation is
strictly parabolic and we get short-time existence from the standard
parabolic theory for quasilinear equations, see e.g.~\cite{ManMar}
for a recent and detailed proof. If we additionally 
assume that $(M,h)=(M,g_0)$ and $\psi(0)=\id_M$,
the flow $\tilde{g}$ which has the same initial data $\tilde{g}(0)
=g_0$ as $g$ is called the Ricci-DeTurck flow \cite{DeTurck:trick}.%
\newline

Now, one can find a solution to the Ricci flow with smooth initial
metric $g(0)=g_0$ as follows. Chose any diffeomorphism
$\psi(0)\colon M\to M$. Since $g(0)$ is smooth, its push-forward
$\tilde{g}(0)$ is also smooth and equation (\ref{2.eq3}) has a
smooth solution for a short time. Next, one computes the vector
field $V$ with (\ref{2.eq2}) and solves the ODE system
\begin{equation*}
\dt \psi = -V \circ \psi.
\end{equation*}
One then recovers $g$ as the pull-back $g=\psi^*\tilde{g}$. This
method also proves uniqueness of the Ricci flow. Indeed, let
$g^1(t)$ and $g^2(t)$ be two solutions of the Ricci flow equation
for $t\in [0,T)$ satisfying $g^1(0)=g^2(0)$. Then one can solve the
harmonic map heat flows $\dt \psi^i=\tau_{g^i}\psi^i$, $i\in\{1,2\}$
with $\psi^1(0)=\psi^2(0)$. This yields two solutions $\tilde{g}^i =
\psi^i_*g^i$ of the dual Ricci-Harmonic map flow with the same
initial values, hence they must agree. Then the corresponding vector
fields $V^i$ agree and the two ODE systems $\dt \psi^i=-V^i
\circ\psi^i$ with the same initial data must have the same solutions
$\psi^1\equiv\psi^2$. Hence also the pull-back metrics $g^1$ and
$g^2$ must agree for all $t\in [0,T)$.%
\newline

For the dual Ricci-Harmonic flow, the evolution equations in 
coordinate form, using only the fixed connection $\hat{\D}$ of
the background metric $h$, are (see e.g.~Simon \cite{Simon:C0metrics})
\begin{equation}\label{2.eq4}
\begin{aligned}\dt \tilde{g}_{ij}&=\tilde{g}^{k\ell}
\hat{\D}_k\hat{\D}_\ell\tilde{g}_{ij}-\tilde{g}^{k\ell}\tilde{g}_{ip}
h^{pq}\hat{R}_{jkq\ell}-\tilde{g}^{k\ell}\tilde{g}_{jp}h^{pq}
\hat{R}_{ikq\ell}\\%
&\quad + \tfrac{1}{2}\tilde{g}^{k\ell}\tilde{g}^{pq}
\big(\hat{\D}_i\tilde{g}_{pk}\hat{\D}_j\tilde{g}_{q\ell}
+2\hat{\D}_k\tilde{g}_{ip}\hat{\D}_q\tilde{g}_{j\ell}
-2\hat{\D}_k\tilde{g}_{ip}\hat{\D}_\ell\tilde{g}_{jq}
-4\hat{\D}_i\tilde{g}_{pk}\hat{\D}_\ell\tilde{g}_{jq}\big),
\end{aligned}
\end{equation}
where $\hat{R}_{ijkl}=(\Rm{h})_{ijkl}$ denotes the Riemannian curvature
tensor of $h$.%
\newline

Recent work of Isenberg, Guenther and Knopf
\cite{IsenbergGuentherKnopf}, Schn\"{u}rer, Schulze and Simon
\cite{SchnuererSchulzeSimon} and others shows that DeTurck's trick
is not only useful to prove short-time existence for the Ricci flow,
but is also useful for convergence and stability results. Their results show that
the $h$-flow itself is also interesting to study. In this article
however, we only use it as a technical tool.

\subsection{Short-time existence and uniqueness for $(RH)_\alpha$}%
Let $(g(t),\phi(t))_{t\in[0,T)}$ be a solution of the $(RH)_\alpha$
flow with initial data $(g(0),\phi(0))=(g_0,\phi_0)$. As for the
Ricci-DeTurck flow above, we now let $\psi(t)\colon(M,g(t))
\to(M,g_0)$ be a solution of the harmonic map heat flow $\dt
\psi=\tau_g\psi$ with $\psi(0)=\id_M$ and denote by
$(\tilde{g}(t),\tilde{\phi}(t))$ the push-forward of
$(g(t),\phi(t))$ with $\psi$. Analogous to formula (\ref{2.eq1})
above, we find
\begin{equation}\label{2.eq5}
\begin{aligned}
\dt \tilde{g}_{ij} &= \psi_*(\dt g)_{ij}+ (\mathcal{L}_V
\tilde{g})_{ij} = -2\tilde{R}_{ij}+2\alpha
\D_i\tilde{\phi}\D_j\tilde{\phi}+ \tilde{\D}_iV_j
+ \tilde{\D}_jV_i,\\%
\dt \tilde{\phi} &= \psi_*(\dt \phi)+ \mathcal{L}_V\tilde{\phi}=
\tau_{\tilde{g}}\tilde{\phi}+\scal{\D\tilde{\phi},V},\end{aligned}
\end{equation}
where $V^\ell=\tilde{g}^{ij}
(\!\phantom{l}^{\tilde{g}}\Gamma^\ell_{ij}-
\!\phantom{l}^{g_0}\Gamma^\ell_{ij})$ and $\tilde{\D}$ denotes the
covariant derivative with respect to $\tilde{g}$. Note that
$\D_i\tilde{\phi}\D_j\tilde{\phi}=(d\tilde{\phi}\otimes
d\tilde{\phi})_{ij}$ as well as
$\scal{\D\tilde{\phi},V}=d\tilde{\phi}(V)$ are independent of the
choice of the metric. Using (\ref{2.eq3}), we find
\begin{equation}\label{2.eq6}
\dt \tilde{g}_{ij} = \tilde{g}^{k\ell}\tilde{\D}_k \hat{\D}_\ell
\tilde{g}_{ij}+2\alpha(\D\tilde{\phi}\otimes\D\tilde{\phi})_{ij},
\end{equation}
which is again quasilinear strictly parabolic. The explicit
evolution equation involving only the fixed Levi-Civita connection
$\hat{\D}$ of $g_0$ can be found from (\ref{2.eq4}) by adding
$2\alpha\hat{\D}_i\tilde{\phi}\hat{\D}_j\tilde{\phi}$ on the right
and replacing $h$ by $g_0$.%
\newline

The evolution equation for $\tilde{\phi}(t)$ in terms of $\hat{\D}$
can be computed as follows. Using normal coordinates on $(N,\gamma)$
we have $\hN\Gamma^\lambda_{\mu\nu}=0$ at the base point and thus
$\tau_{\tilde{g}}\tilde{\phi}=\Lap_{\tilde{g}}\tilde{\phi}$. We find
\begin{equation}\label{2.eq7}
\begin{aligned}
\dt \tilde{\phi}^\lambda &= \Lap_{\tilde{g}}\tilde{\phi}^\lambda +
\scal{\D\tilde{\phi}^\lambda,V} =
\tilde{g}^{k\ell}\big(\partial_k\partial_\ell\tilde{\phi}^\lambda-
\!\phantom{l}^{\tilde{g}}\Gamma^j_{k\ell}\D_j\tilde{\phi}^\lambda\big)
+\D_j\tilde{\phi}^\lambda V^j\\%
&= \tilde{g}^{k\ell}\big(\hat{\D}_k\hat{\D}_\ell\tilde{\phi}^\lambda
+\!\phantom{l}^{g_0}\Gamma^j_{k\ell}\D_j\tilde{\phi}^\lambda
-\!\phantom{l}^{\tilde{g}}\Gamma^j_{k\ell}\D_j\tilde{\phi}^\lambda\big)+
\D_j\tilde{\phi}^\lambda\cdot\tilde{g}^{k\ell}
\big(\!\phantom{l}^{\tilde{g}}\Gamma^j_{k\ell}-
\!\phantom{l}^{g_0}\Gamma^j_{k\ell}\big)\\%
&= \tilde{g}^{k\ell}\hat{\D}_k\hat{\D}_\ell\tilde{\phi}^\lambda.
\end{aligned}
\end{equation}
Putting these results together, we have proved the following.
\begin{prop}\label{2.prop1}
Let $(g(t),\phi(t))$ be a solution of the $(RH)_\alpha$ flow with
initial data $(g(0),\phi(0))=(g_0,\phi_0)$. Let
$\psi(t)\colon(M,g(t)) \to(M,g_0)$ solve the harmonic map heat flow
$\dt \psi=\tau_g\psi$ with $\psi(0)=\id_M$ and let
$(\tilde{g}(t),\tilde{\phi}(t))$ denote the push-forward of
$(g(t),\phi(t))$ with $\psi$. Let $\hat{\D}$ be the (fixed)
Levi-Civita connection with respect to $g_0$. Then the dual
$(RH)_\alpha$ flow $(\tilde{g}(t),\tilde{\phi}(t))$ satisfies
\begin{align*}
\dt \tilde{g}_{ij}&=\tilde{g}^{k\ell}
\hat{\D}_k\hat{\D}_\ell\tilde{g}_{ij}-\tilde{g}^{k\ell}\tilde{g}_{ip}
g_0^{pq}\hat{R}_{jkq\ell}-\tilde{g}^{k\ell}\tilde{g}_{jp}g_0^{pq}
\hat{R}_{ikq\ell}+2\alpha\hat{\D}_i\tilde{\phi}\hat{\D}_j\tilde{\phi}\\%
&\quad + \tfrac{1}{2}\tilde{g}^{k\ell}\tilde{g}^{pq}
\big(\hat{\D}_i\tilde{g}_{pk}\hat{\D}_j\tilde{g}_{q\ell}
+2\hat{\D}_k\tilde{g}_{ip}\hat{\D}_q\tilde{g}_{j\ell}
-2\hat{\D}_k\tilde{g}_{ip}\hat{\D}_\ell\tilde{g}_{jq}
-4\hat{\D}_i\tilde{g}_{pk}\hat{\D}_\ell\tilde{g}_{jq}\big),\\%
\dt \tilde{\phi}^\lambda &= \tilde{g}^{k\ell}\hat{\D}_k\hat{\D}_\ell
\tilde{\phi}^\lambda +
\tilde{g}^{k\ell}\big(\hN\Gamma^\lambda_{\mu\nu}\circ\phi\big)
\hat{\D}_k\tilde{\phi}^\mu\hat{\D}_\ell\tilde{\phi}^\nu.
\end{align*}
In particular, the principal symbol for both equations is
$\sigma(\xi)=\tilde{g}^{ij}\xi_i \xi_j\cdot \id$, i.e.~the
push-forward flow is a solution to a system of strictly parabolic
equations.
\end{prop}
Short-time existence and uniqueness for the dual flow (and hence
also for the $(RH)_\alpha$ flow itself) now follow exactly as in the
simpler case of the Ricci and the dual Ricci-Harmonic flow described
above.%

\subsection{Evolution equations for $R$, $\mathrm{Rc}$,
$\abs{\D\phi}^2$ and $\D\phi\otimes\D\phi$}%
In the following, we often use commutator identities on bundles like 
$T^*M \otimes \phi^*TN$. The necessary formulas are collected in the 
appendix. We denote the Riemannian curvature tensor on $(N,\gamma)$ 
by $\NRm{}$ and let $\Rm{}$, $\Rc{}$ and $R$ be the Riemannian, Ricci 
and scalar curvature on $(M,g)$. Moreover, we write
\begin{align*}
\scal{\Rc{},\D\phi\otimes\D\phi}&:=R_{ij}\D_i\phi^\kappa\D_j\phi^\kappa,\\%
\Scal{\!\NRm{\D_i\phi,\D_j\phi}\D_j\phi, \D_i\phi} &:= \hN
\!R_{\kappa\mu\lambda\nu}\D_i
\phi^\kappa\D_j\phi^\mu\D_i\phi^\lambda\D_j\phi^\nu.
\end{align*}
Finally, we use the fact that $\tau_g\phi=\D_p\D_p\phi$ for the
covariant derivative $\D$ on $T^*M \otimes \phi^*TN$, cf.~Jost
\cite[Section 8.1]{Jost:Riemannian}. From the commutator identities
in the appendix, we immediately obtain for $\phi \in C^\infty(M,N)$
\begin{equation}\label{4.Lap}
\begin{split}
\Lap_g(\D_i\phi\D_j\phi) &=
\D_i\tau_g\phi\D_j\phi+\D_i\phi\D_j\tau_g\phi +
2\D_i\D_p\phi\D_j\D_p\phi\\%
&\quad\,+ R_{ip}\D_p\phi\D_j\phi + R_{jp}\D_p\phi\D_i\phi -
2\Scal{\!\NRm{\D_i\phi,\D_p\phi}\D_p\phi,\D_j\phi}.
\end{split}
\end{equation}
\begin{rem}
Taking the trace and using $\tau_g\phi=0$, we find the well-known
Bochner identity for harmonic maps, cf.~Jost \cite[Section
8.7]{Jost:Riemannian},
\begin{equation*}
-\Lap_g \abs{\D\phi}^2 + 2\abs{\D^2\phi}^2 +
2\scal{\Rc{},\D\phi\otimes\D\phi} =
2\Scal{\!\NRm{\D_i\phi,\D_j\phi}\D_j\phi,\D_i\phi}.
\end{equation*}
\end{rem}
Now, we compute the evolution equations for the scalar and Ricci
curvature on $M$.
\begin{prop}\label{2.prop3}
Let $(g(t),\phi(t))$ be a solution to the $(RH)_\alpha$ flow
equation. Then the scalar curvature evolves according to
\begin{equation}\label{2.eq9}
\begin{aligned}
\dt R &= \Lap R + 2\abs{\Rc{}}^2-4\alpha
\scal{\Rc{},\D\phi\otimes\D\phi}
+2\alpha\abs{\tau_g\phi}^2-2\alpha\abs{\D^2\phi}^2\\
&\quad\,+ 2\alpha \Scal{\!\NRm{\D_i\phi,\D_j\phi} \D_j\phi,\D_i\phi}
\end{aligned}
\end{equation}
and the Ricci curvature evolves by
\begin{equation}\label{2.eq10}
\begin{aligned}
\dt R_{ij} &= \Lap_L R_{ij}-2R_{iq}R_{jq}+2R_{ipjq}R_{pq}
+2\alpha\,\tau_g\phi\D_i\D_j\phi-2\alpha\D_p\D_i\phi\D_p\D_j\phi\\
&\quad\,+2\alpha R_{pijq}\D_p\phi\D_q\phi +2\alpha
\Scal{\!\NRm{\D_i\phi,\D_p\phi}\D_p\phi,\D_j\phi}.
\end{aligned}
\end{equation}
Here, $\Lap_L$ denotes the Lichnerowicz Laplacian, introduced in
\cite{Lichnerowicz}, which is defined on symmetric two-tensors
$t_{ij}$ by
\begin{equation*}
\Lap_L t_{ij} := \Lap
t_{ij}+2R_{ipjq}t_{pq}-R_{ip}t_{pj}-R_{jp}t_{pi}.
\end{equation*}
\end{prop}
\begin{proof}
We know that for $\dt g_{ij} = h_{ij}$ the evolution equation for
the Ricci tensor is given by
\begin{equation}\label{2.eq11}
\dt R_{ij} = -\tfrac{1}{2}\Lap_L h_{ij}+\tfrac{1}{2}\D_i\D_p
h_{pj}+\tfrac{1}{2}\D_j\D_p h_{pi}-\tfrac{1}{2}\D_i\D_j(\tr_g h),
\end{equation}
see for example \cite[Proposition 1.4]{Muller:Harnack} for a
proof of this general variation formula. For $h_{ij}=-2R_{ij}$, we
obtain (with the twice contracted second Bianchi identity) 
$\dt R_{ij}=\Lap_L R_{ij}$. For $h_{ij}=2\D_i\phi\D_j\phi$, we 
compute, using (\ref{4.Lap})
\begin{align*}
\dt R_{ij}&=-\Lap_L(\D_i\phi\D_j\phi)+\D_i\D_p(\D_p\phi\D_j\phi)+
\D_j\D_p(\D_p\phi\D_i\phi)-\D_i\D_j(\D_p\phi\D_p\phi)\\%
&= 2\tau_g\phi\D_i\D_j\phi-2\D_i\D_p\phi\D_j\D_p\phi
-2R_{ipjq}\D_p\phi\D_q\phi
+2\Scal{\!\NRm{\D_i\phi,\D_p\phi}\D_p\phi,\D_j\phi}.
\end{align*}
Linearity then yields (\ref{2.eq10}). For the evolution equation for
$R$, use
\begin{equation*}
\dt R =\dt (g^{ij}R_{ij})=g^{ij}(\dt
R_{ij})+(2R_{ij}-2\alpha\D_i\phi\D_j\phi)R_{ij}.
\end{equation*}
The desired evolution equation (\ref{2.eq9}) follows. An alternative
and more detailed proof can be found in the author's thesis
\cite[Proposition 2.3]{Muller:diss}.
\end{proof}
Next, we compute the evolution equations for $\abs{\D\phi}^2$ and
$\D\phi\otimes\D\phi$.
\begin{prop}\label{2.prop4}
Let $(g(t),\phi(t))$ be a solution of\/ $(RH)_\alpha$. Then the
energy density of $\phi$ satisfies the evolution equation
\begin{equation}\label{2.eq12}
\dt\abs{\D\phi}^2=\Lap\abs{\D\phi}^2-2\alpha\abs{\D\phi\otimes\D\phi}^2
-2\abs{\D^2\phi}^2 +2\Scal{\!\NRm{\D_i\phi,\D_j\phi}
\D_j\phi,\D_i\phi}.
\end{equation}
Furthermore, we have
\begin{equation}\label{2.eq13}
\begin{aligned}
\dt (\D_i\phi\D_j\phi) &=\Lap(\D_i\phi\D_j\phi)-
2\D_p\D_i\phi\D_p\D_j\phi-R_{ip}\D_p\phi\D_j\phi-
R_{jp}\D_p\phi\D_i\phi\\
&\quad\,+ 2\Scal{\!\NRm{\D_i\phi,\D_p\phi}\D_p\phi,\D_j\phi}.
\end{aligned}
\end{equation}
\end{prop}
\begin{proof}
We start with the second statement. We have
\begin{align*}
\dt(\D_i\phi\D_j\phi) &= (\D_t\D_i\phi)\D_j\phi
+(\D_t\D_j\phi)\D_i\phi\\%
&=\D_i(\dt\phi)\D_j\phi + \D_j(\dt\phi)\D_i\phi\\%
&=\D_i\tau_g\phi\D_j\phi+\D_j\tau_g\phi\D_i\phi,
\end{align*}
where the meaning of the covariant time derivative $\D_t$ is
explained in the appendix. The desired evolution equation
(\ref{2.eq13}) now follows directly from (\ref{4.Lap}). We
obtain (\ref{2.eq12}) from (\ref{2.eq13}) by taking the trace,
\begin{align*}
\dt\abs{\D\phi}^2 &= (2R_{ij}-2\alpha\D_i\phi\D_j\phi)
\D_i\phi\D_j\phi + g^{ij}\dt(\D_i\phi\D_j\phi)\\
&=-2\alpha\abs{\D\phi\otimes\D\phi}^2+ \Lap\abs{\D\phi}^2
-2\abs{\D^2\phi}^2 +2\Scal{\!\NRm{\D_i\phi,\D_j\phi}
\D_j\phi,\D_i\phi}.\qedhere
\end{align*}
\end{proof}

\subsection{Evolution of $\sS=\mathrm{Rc}-\alpha\D\phi\otimes\D\phi$
and its trace}%
We write again $\sS:=\Rc{}-\alpha\D\phi\otimes\D\phi$ with
components $S_{ij}=R_{ij}- \alpha\D_i\phi\D_j\phi$ and let 
$S= R-\alpha\abs{\D\phi}^2$ be its trace. Then, we can write
the $(RH)_\alpha$ flow as
\begin{equation*}
\left\{\begin{aligned} \dt g_{ij} &= -2S_{ij},\\
\dt \phi &= \tau_g\phi,\end{aligned}\right.
\end{equation*}
and the energy from Section 3 as $\sF_\alpha(g,\phi,f) := \int_M \big(S+\abs{\D f}^2_g\big)e^{-f}dV_g$. It is thus convenient to study the 
evolution equations for $\sS$ and
$S$. Indeed, many terms cancel and we get much nicer equations than
in the previous subsection.
\begin{thm}\label{2.thm5}
Let $(g(t),\phi(t))$ solve $(RH)_\alpha$ with $\alpha(t) \equiv
\alpha > 0$. Then $\sS$ and\/ $S$ defined as above satisfy the
following evolution equations
\begin{equation}\label{2.eq16}
\begin{aligned}
\dt S&= \Lap S + 2\abs{S_{ij}}^2+2\alpha \abs{\tau_g\phi}^2,\\
\dt S_{ij}&= \Lap_L S_{ij} + 2\alpha \,\tau_g\phi\D_i\D_j\phi.
\end{aligned}
\end{equation}
\end{thm}
\begin{proof}
This follows directly by combining the evolution equations from
Proposition \ref{2.prop3} with those from Proposition \ref{2.prop4}.
\end{proof}
\begin{rem}
Note that in contrast to the evolution of $\Rc{}$, $R$,
$\D\phi\otimes \D\phi$ and $\abs{\D\phi}^2$ the evolution equations
in Theorem \ref{2.thm5} for the combinations
$\Rc{}-\alpha\D\phi\otimes\D\phi$ and $R-\alpha\abs{\D\phi}^2$ do
\emph{not} depend on the intrinsic curvature of $N$.
\end{rem}
\begin{cor}\label{2.cor6}
For a solution $(g(t),\phi(t))$ of\/ $(RH)_\alpha$ with a
time-dependent coupling function $\alpha(t)$, we get
\begin{equation}\label{2.eq17}
\begin{aligned}
\dt S&= \Lap S + 2\abs{S_{ij}}^2+2\alpha \abs{\tau_g\phi}^2
-\dot{\alpha}\abs{\D\phi}^2,\\
\dt S_{ij}&= \Lap_L S_{ij} + 2\alpha \,\tau_g\phi\D_i\D_j\phi
-\dot{\alpha}\D_i\phi\D_j\phi.
\end{aligned}
\end{equation}
\end{cor}

\section{First results about singularities}%
In this section, we often use the weak maximum principle which
states that for parabolic partial differential equations with a
reaction term a solution of the corresponding ODE yields 
pointwise bounds for the solutions of the PDE. Since
we work on an evolving manifold, we need a slightly generalized
version. The following result is proved in \cite[Theorem
4.4]{RF:intro}.
\begin{prop}\label{app.prop1}
Let $u:M\times[0,T)\to \RR$ be a smooth function satisfying
\begin{equation}\label{app.eq15}
\dt u \geq \Lap_{g(t)}u + \scal{X(t),\D u}_{g(t)}+F(u),
\end{equation}
where $g(t)$ is a smooth 1-parameter family of metrics on $M$,
$X(t)$ a smooth 1-parameter family of vector fields on $M$, and
$F:\RR\to\RR$ is a locally Lipschitz function. Suppose that
$u(\cdot,0)$ is bounded below by a constant $C_0\in\RR$ and let
$\phi(t)$ be a solution to
\begin{equation*}
\dt \phi = F(\phi), \quad \phi(0)=C_0.
\end{equation*}
Then $u(x,t)\geq \phi(t)$ for all $x\in M$ and all $t\in[0,T)$ for
which $\phi(t)$ exists.
\end{prop}
Similarly, if (\ref{app.eq15}) is replaced by $\dt u \leq \Lap_{g(t)}u 
+\scal{X(t),\D u}_{g(t)}+F(u)$ and $u(\cdot,0)$ is bounded from above 
by $C_0$, then $u(x,t)\leq \phi(t)$ for all $x\in M$ and $t\in[0,T)$ 
for which the solution $\phi(t)$ of the corresponding ODE exists.%
\newline

Using this, an immediate consequence of Corollary \ref{2.cor6} is
the following.
\begin{cor}\label{2.cor7}
Let $(g(t),\phi(t))$ be a solution to the $(RH)_\alpha$ flow with a
nonnegative, non-increasing coupling function $\alpha(t)$. Let
$S(t)=R(g(t))-\alpha(t)\abs{\D\phi(t)}^2_{g(t)}$ as above, with
initial data $S(0)>0$ on $M$. Then $R_{min}(t):=\min_{x\in M}R(x,t)
\to \infty$ in finite time and thus $g(t)$ must become singular in
finite time $T_{sing} \leq \frac{m}{2S_{min}(0)} < \infty$.
\end{cor}
\begin{proof}
Since $\alpha(t) \geq 0$ and $\dot{\alpha}(t) \leq 0$ for all $t\geq
0$, Corollary \ref{2.cor6} yields
\begin{equation}\label{2.eq18}
\dt S \geq \Lap S +2\abs{S_{ij}}^2 \geq \Lap S + \tfrac{2}{m}S^2
\end{equation}
and thus by comparing with solutions of the ODE $\tfrac{d}{dt} a(t)
= \frac{2}{m}a(t)^2$ which are
\begin{equation*}
a(t) = \frac{a(0)}{1-\tfrac{2t}{m}a(0)},
\end{equation*}
the maximum principle above, yields
\begin{equation}\label{2.eq19}
S_{min}(t) \geq \frac{S_{min}(0)}{1-\tfrac{2t}{m}S_{min}(0)}
\end{equation}
for all $t \geq 0$ as long as the flow exists. In particular, if
$S_{min}(0)>0$ this implies that $S_{min}(t)\to\infty$ in finite
time $T_0 \leq \frac{m}{2S_{min}(0)}<\infty$. Since
$R=S+\alpha\abs{\D\phi}^2\geq S$, we find that also $R_{min}(t)
\to\infty$ before $T_0$ and thus $g(t)$ has to become singular in
finite time $T_{sing}\leq T_0 \leq \frac{m}{2S_{min}(0)}<\infty$.
\end{proof}
As a second consequence, we see that if the energy density $e(\phi)
= \frac{1}{2}\abs{\D\phi}^2$ blows up at some point in space-time
while $\alpha(t)$ is bounded away from zero, then also $g(t)$ must
become singular at this point.
\begin{cor}\label{2.cor8}
Let $(g(t),\phi(t))_{t\in[0,T)}$ be a smooth solution of\/
$(RH)_\alpha$ with a non-increasing coupling function $\alpha(t)$
satisfying $\alpha(t)\geq \underaccent{\bar}\alpha >0$ for all
$t\in[0,T)$. Suppose that $\abs{\D\phi}^2(x_k,t_k) \to \infty$ for a
sequence $(x_k,t_k)_{k\in\mathbb{N}}$ with $t_k \nearrow T$. Then
also $R(x_k,t_k)\to \infty$ for this sequence and thus $g(t_k)$ must
become singular as $t_k$ approaches $T$.
\end{cor}
\begin{proof}
From (\ref{2.eq18}) we obtain $S\geq S_{min}(0)$ and thus
\begin{equation}\label{2.eq20}
R = \alpha\abs{\D\phi}^2 + S \geq
\underaccent{\bar}\alpha\abs{\D\phi}^2 + S_{min}(0), \quad \forall
(x,t)\in M\times[0,T).
\end{equation}
Hence, if $\abs{\D\phi}^2(x_k,t_k)\to\infty$ for a sequence
$(x_k,t_k)_{k\in\mathbb{N}}\subset M\times[0,T)$ with $t_k\nearrow
T$ then also $R(x_k,t_k)\to\infty$ for this sequence and $g(t_k)$
must become singular as $t_k\nearrow T$.
\end{proof}
\begin{rem}
The proof shows that Corollary \ref{2.cor8} stays true if $\alpha(t)
\searrow 0$ as $t \nearrow T$ as long as $\abs{\D\phi}^2(x_k,t_k)
\to \infty$ fast enough such that $\alpha(t_k)\abs{\D\phi}^2
(x_k,t_k) \to \infty$ still holds true.
\end{rem}
\vspace{3mm}%
Now, we derive for $t>0$ an improved version of (\ref{2.eq20}) which
does not depend on the initial data $S(0)$. Using (\ref{2.eq18}) and
the maximum principle, we see that if $S_{min}(0)\geq C\in \RR$ we
obtain
\begin{equation*}
S_{min}(t)\geq \frac{C}{1-\tfrac{2t}{m}C} \longrightarrow -
\frac{m}{2t} \qquad(C\to-\infty)
\end{equation*}
and thus $S(t) \geq -\frac{m}{2t}$ for all $t> 0$ as long as the
flow exists, independent of $S(0)$. More rigorously, this is
obtained as follows. The inequality (\ref{2.eq18}) implies
\begin{equation*}
\dt (tS) = S + t\big(\dt S\big) \geq \Lap(tS) +
S\big(1+\tfrac{2t}{m}S\big).
\end{equation*}
If $(x_0,t_0)$ is a point where $tS$ first reaches its minimum over
$M\times[0,T-\delta]$, $\delta>0$ arbitrarily small, we get
$S(x_0,t_0)\big(1+\tfrac{2t_0}{m}S(x_0,t_0))\leq 0$, which is only
possible for $t_0S(x_0,t_0)\geq -\tfrac{m}{2}$. Hence
$tS\geq-\frac{m}{2}$ on all of $M\times[0,T-\delta]$. Since $\delta$
was arbitrary, we obtain the desired inequality $S(t) \geq
-\frac{m}{2t}$ everywhere on $M\times(0,T)$. This yields
\begin{equation*}
R\geq \alpha\abs{\D\phi}^2 -\tfrac{m}{2t} \geq
\underaccent{\bar}\alpha\abs{\D\phi}^2 -\tfrac{m}{2t}, \quad \forall
(x,t)\in M\times(0,T),
\end{equation*}
which immediately implies the following converse of Corollary
\ref{2.cor8}.
\begin{cor}\label{2.cor9}
Let $(g(t),\phi(t))_{t\in[0,T)}$ be a smooth solution of\/
$(RH)_\alpha$ with a non-increasing coupling function $\alpha(t)\geq
\underaccent{\bar}\alpha>0$ for all\/ $t\in[0,T)$. Assume that
$R\leq R_0$ on $M\times[0,T)$. Then
\begin{equation}\label{2.eq21}
\abs{\D\phi}^2 \leq
\frac{R_0}{\underaccent{\bar}\alpha}+\frac{m}{2\underaccent{\bar}\alpha
t}, \quad \forall (x,t)\in M\times(0,T).
\end{equation}
\end{cor}
Singularities of the type as in Corollary \ref{2.cor8}, where the
energy density of $\phi$ blows up, can not only be ruled out if the
curvature of $M$ stays bounded. There is also a way to rule them out
\emph{a-priori}. Namely, such singularities cannot form if either $N$ has
non-positive sectional curvatures or if we choose the coupling
constants $\alpha(t)$ large enough such that
\begin{equation*}
\max_{y \in N}\!\hN K(y) \leq \frac{\alpha}{m}.
\end{equation*}
Here $\hN K$ denotes the sectional curvature of $N$. More precisely, we
have the following estimates for the energy density
$\abs{\D\phi}^2$.
\begin{prop}\label{2.prop10}
Let $(g(t),\phi(t))_{t\in[0,T)}$ be a solution of\/ $(RH)_\alpha$
with a non-increasing $\alpha(t) \geq 0$ and let
the sectional curvature of\/ $N$ be bounded above by $\!\hN K\leq
c_0$. Then
\begin{itemize}
\item[i)] if $N$ has non-positive sectional curvatures or more
generally if\/ $c_0 - \frac{\alpha(t)}{m} \leq 0$, the energy
density of $\phi$ is bounded by its initial data,
\begin{equation}\label{2.eq22}
\abs{\D\phi(x,t)}^2 \leq \max_{y\in M}\abs{\D\phi(y,0)}^2, \quad
\forall (x,t)\in M\times[0,T).
\end{equation}
\item[ii)] if $N$ has non-positive sectional curvatures and
$\alpha(t)\geq \underaccent{\bar}\alpha>0$, we have in addition to
(\ref{2.eq22}) the estimate
\begin{equation}\label{2.eq23}
\abs{\D\phi(x,t)}^2 \leq \frac{m}{2\underaccent{\bar}\alpha t},
\quad \forall (x,t)\in M\times(0,T).
\end{equation}
\item[iii)] in general, the energy density satisfies
\begin{equation}\label{2.eq24}
\abs{\D\phi(x,t)}^2 \leq 2\max_{y\in M}\abs{\D\phi(y,0)}^2, \quad
\forall (x,t)\in M\times[0,T^*),
\end{equation}
where $T^*:= \min\big\{T,\big(4c_0\max_{y\in
M}\abs{\D\phi(y,0)}^2\big)\!\!\phantom{.}^{-1}\big\}$.
\end{itemize}
\end{prop}
\begin{proof}
This is a consequence of the evolution equation (\ref{2.eq12}) and
the Cauchy-Schwarz inequality
\begin{equation}\label{2.eq25}
\tfrac{1}{m}\abs{\D\phi}^4 =
\tfrac{1}{m}\abs{g^{ij}\D_i\phi\D_j\phi}^2 \leq
\abs{\D_i\phi\D_j\phi}^2 \leq \abs{\D\phi}^4.
\end{equation}
\begin{itemize}
\item[i)] If $N$ has non-positive sectional curvatures,
$\Scal{\!\NRm{\D_i\phi,\D_j\phi}\D_j\phi,\D_i\phi} \leq 0$, the
evolution equation (\ref{2.eq12}) implies
\begin{equation}\label{2.eq26}
\dt \abs{\D\phi}^2 \leq \Lap \abs{\D\phi}^2.
\end{equation}
If $c_0 -\frac{\alpha(t)}{m}\leq 0$, we have
$2\Scal{\!\NRm{\D_i\phi,\D_j\phi}\D_j\phi,\D_i\phi} \leq
2c_0\abs{\D\phi}^4\leq 2\tfrac{\alpha}{m}\abs{\D\phi}^4\leq
2\alpha\abs{\D_i\phi\D_j\phi}^2$, and we get again (\ref{2.eq26})
from (\ref{2.eq12}). The claim now follows from the maximum
principle applied to (\ref{2.eq26}).%
\item[ii)] If $\Scal{\!\NRm{\D_i\phi,\D_j\phi}\D_j\phi,\D_i\phi} \leq
0$, (\ref{2.eq12}) and (\ref{2.eq25}) imply
\begin{equation*}
\dt\abs{\D\phi}^2\leq \Lap
\abs{\D\phi}^2-2\alpha\abs{\D_i\phi\D_j\phi}^2\leq
\Lap\abs{\D\phi}^2-2\tfrac{\alpha}{m}\abs{\D\phi}^4.
\end{equation*}
We obtain
\begin{equation*}
\dt \big(t\abs{\D\phi}^2\big)=\abs{\D\phi}^2 +t\big(\dt
\abs{\D\phi}^2\big) \leq \Lap\big(t\abs{\D\phi}^2\big) +
\abs{\D\phi}^2\big(1-2t\tfrac{\alpha}{m}\abs{\D\phi}^2\big).
\end{equation*}
At the first point $(x_0,t_0)$ where $t\abs{\D\phi}^2$ reaches its
maximum over $M\times[0,T-\delta]$, $\delta>0$ arbitrary, we find
$1-2t_0\frac{\alpha}{m}\abs{\D\phi}^2(x_0,t_0)\geq 0$, i.e.
\begin{equation*}
t_0\abs{\D\phi}^2(x_0,t_0) \leq \frac{m}{2\alpha}\leq
\frac{m}{2\underaccent{\bar}\alpha},
\end{equation*}
which implies that
$t\abs{\D\phi}^2\leq\frac{m}{2\underaccent{\bar}\alpha}$ for
every $(x,t)\in M\times[0,T-\delta]$. The claim follows.%
\item[iii)] From (\ref{2.eq12}), we get
\begin{equation*}
\dt \abs{\D\phi}^2 \leq \Lap\abs{\D\phi}^2 + 2c_0\abs{\D\phi}^4.
\end{equation*}
By comparing with solutions of the ODE $\tfrac{d}{dt} a(t) = 2c_0
a(t)^2$, which are
\begin{equation*}
a(t) = \frac{a(0)}{1-2c_0a(0)t}, \quad t \leq \frac{1}{2c_0a(0)},
\end{equation*}
the maximum principle from Proposition \ref{app.prop1} implies
\begin{equation}\label{2.eq27}
\abs{\D\phi(x,t)}^2 \leq \frac{\max_{y\in M}\abs{\D\phi(y,0)}^2}{1
-2c_0\max_{y\in M}\abs{\D\phi(y,0)}^2\, t},
\end{equation}
for all $x\in M$ and $t\leq \min\big\{T,\big(2c_0\max_{y\in
M}\abs{\D\phi(y,0)}^2\big)\!\!\phantom{.}^{-1}\big\}$. In
particular, this proves the doubling-time estimate that we
claimed.\qedhere
\end{itemize}
\end{proof}

\section{Gradient estimates and long-time existence}
For solutions $(g(t),\phi(t))$ of the $(RH)_\alpha$ flow with
non-increasing $\alpha(t)\geq\underaccent{\bar}\alpha>0$, we have
seen in Corollary \ref{2.cor9} that a uniform bound on the curvature
of $(M,g(t))$ implies a uniform bound on $\abs{\D\phi}^2$.
Therefore, one expects that a uniform curvature bound suffices to
show long-time existence for our flow. The proof of this result is
the main goal of this section.

\subsection{Evolution equations for $\mathrm{Rm}$ and $\D^2\phi$}
With $\dt g_{ij} = h_{ij} :=
-2R_{ij}+2\alpha\D_i\phi^\mu\D_j\phi^\mu$, we find the evolution
equation for the Christoffel symbols
\begin{equation}\label{3.eq1}
\begin{aligned}
\dt\Gamma_{ij}^p &=\tfrac{1}{2}g^{pq}(\D_ih_{jq}+\D_jh_{iq}
-\D_qh_{ij})\\%
&=g^{pq}(-\D_iR_{jq}-\D_jR_{iq}+\D_qR_{ij})
+2\alpha\D_i\D_j\phi\D^p\phi
\end{aligned}
\end{equation}
With this, an elementary computation yields the following evolution equation
for the Riemannian curvature tensor (see \cite[Proposition 3.2]{Muller:diss} 
for a detailed proof).
\begin{prop}\label{3.prop2}
Let $(g(t),\phi(t))_{t\in [0,T)}$ be a solution of\/ $(RH)_\alpha$.
Then the Riemann tensor satisfies
\begin{equation}\label{3.eq4}
\begin{aligned}
\dt R_{ijk\ell} &= \D_i\D_kR_{j\ell} -\D_i\D_\ell R_{jk}
-\D_j\D_kR_{i\ell} + \D_j\D_\ell R_{ik}-R_{ijq\ell}R_{kq}
-R_{ijkq}R_{\ell q}\\%
&\quad\,+2\alpha\big(\D_i\D_k\phi\D_j\D_\ell\phi
-\D_i\D_\ell\phi\D_j\D_k\phi -\Scal{\!\NRm{\D_i\phi,\D_j\phi}
\D_k\phi,\D_\ell\phi}\big).
\end{aligned}
\end{equation}
\end{prop}
\begin{rem}
Taking the trace of (\ref{3.eq4}), we obtain (\ref{2.eq10}), using
the twice traced second Bianchi identity. This gives
an alternative proof of Proposition \ref{2.prop3}.
\end{rem}
If we set $\alpha = 0$ in (\ref{3.eq4}), we obtain the evolution
equation for the curvature tensor under the Ricci flow. It is
well-known that this evolution equation can be written in a nicer
form, in which its parabolic nature is more apparent. In \cite[Lemma
7.2]{Hamilton:3folds}, Hamilton proved
\begin{equation}\label{3.eq5}
\begin{aligned}
\D_i\D_kR_{j\ell}-\D_i&\D_\ell R_{jk}-\D_j\D_kR_{i\ell}+\D_j\D_\ell
R_{ik}\\%
&= \Lap R_{ijk\ell}+2(B_{ijk\ell}-B_{ij\ell k}-B_{i\ell jk}
+B_{ikj\ell}) -R_{pjk\ell}R_{pi}-R_{ipk\ell}R_{pj},
\end{aligned}
\end{equation}
where $B_{ijk\ell} := R_{ipjq}R_{kp\ell q}$. Plugging this into
(\ref{3.eq4}) yields the following corollary.
\begin{cor}\label{3.cor3}
Along the $(RH)_\alpha$ flow, the Riemannian curvature tensor
evolves by
\begin{equation}\label{3.eq6}
\begin{aligned}
\dt R_{ijk\ell} &= \Lap R_{ijk\ell} + 2(B_{ijk\ell}-B_{ij\ell
k}-B_{i\ell jk} +B_{ikj\ell})\\%
&\quad -(R_{pjk\ell}R_{pi}+R_{ipk\ell}R_{pj}
+R_{ijp\ell}R_{pk}+R_{ijkp}R_{p\ell})\\%
&\quad +2\alpha\big(\D_i\D_k\phi\D_j\D_\ell\phi-
\D_i\D_\ell\phi\D_j\D_k\phi -\Scal{\!\NRm{\D_i\phi,\D_j\phi}
\D_k\phi,\D_\ell\phi}\big).
\end{aligned}
\end{equation}
\end{cor}
There is a useful convention for writing such equations in a short
form.
\begin{defn}\label{3.defn1}
For two quantities $A$ and $B$, we denote by $A*B$ any quantity
obtained from $A \otimes B$ by summation over pairs of matching
(Latin and Greek) indices, contractions with the metrics $g$ and
$\gamma$ and their inverses, and multiplication with constants
depending only on $m=\dim M$, $n=\dim N$ and the ranks of $A$ and
$B$. We also write $(A)^{*1}:=1*A$, $(A)^{*2}=A*A$, etc.
\end{defn}
This notation allows us to write (\ref{3.eq1}) and (\ref{3.eq6}) in
the short forms
\begin{equation}\label{3.eq3}
\dt \Gamma = (\D\Rm{})^{*1} + \alpha\D^2\phi*\D\phi.
\end{equation}
and
\begin{equation}\label{3.eq7}
\dt \Rm{} = \Lap \Rm{} + (\Rm{})^{*2} + \alpha (\D^2\phi)^{*2} +
\alpha \NRm{}*(\D\phi)^{*4}.
\end{equation}
It is now easy to compute the evolution of the length of the Riemann
tensor. Together with $\dt g^{-1} = (\Rm{})^{*1}+
\alpha(\D\phi)^{*2}$, the above formula yields
\begin{equation}\label{3.eq8}
\begin{aligned}
\dt \abs{\Rm{}}^2 &= \big(\dt g^{-1}\big)*\Rm{}*\Rm{}
+ 2R_{ijk\ell}\big(\dt R_{ijk\ell}\big)\\%
&= \Lap \abs{\Rm{}}^2 - 2\abs{\D\Rm{}}^2 + (\Rm{})^{*3} +
\alpha(\Rm{})^{*2}*(\D\phi)^{*2}\\%
&\quad + \alpha \Rm{}*(\D^2\phi)^{*2} + \alpha
\Rm{}*\NRm{}*(\D\phi)^{*4}.
\end{aligned}
\end{equation}
\begin{cor}\label{3.cor4}
Along the $(RH)_\alpha$ flow, the Riemannian curvature tensor
satisfies
\begin{equation}\label{3.eq9}
\begin{aligned}
\dt \abs{\Rm{}}^2 &\leq \Lap \abs{\Rm{}}^2 - 2\abs{\D\Rm{}}^2
+C \abs{\Rm{}}^3 + \alpha C \abs{\D\phi}^2\abs{\Rm{}}^2\\%
&\quad + \alpha C \abs{\D^2\phi}^2\abs{\Rm{}} + \alpha Cc_0
\abs{\D\phi}^4\abs{\Rm{}},
\end{aligned}
\end{equation}
for constants $C\geq 0$ depending only on the dimension of $M$ and
$c_0=c_0(N)\geq 0$ depending only on the curvature of $N$. If $N$ is
flat, we can choose $c_0 = 0$.
\end{cor}
\begin{proof}
Follows directly from (\ref{3.eq8}) and the fact that $\abs{\NRm{}}$
is bounded on compact $N$.
\end{proof}
For the evolution equation for the Hessian of $\phi$, it is
important that we do not use the $*$-notation directly. Indeed, we
will see that all the terms containing derivatives of the curvature
of $M$ cancel each other (using the second Bianchi identity), a
phenomenon which cannot be seen when working with the $*$-notation.%
\newline

A short computation using (\ref{app.eq7}) and (\ref{app.eq8}) shows that 
the commutator $[\D_i\D_j,\Lap]\phi^\lambda =
\D_i\D_j\tau_g\phi^\lambda-\Lap\D_i\D_j\phi^\lambda$ is given by
\begin{equation}\label{3.eq10}
\begin{aligned}
{[\D_i\D_j,\Lap]}\phi^\lambda &= \D_kR_{jpik}\D_p\phi^\lambda
+2R_{ikjp}\D_k\D_p\phi^\lambda\\%
&\quad\,-R_{ip}\D_j\D_p\phi^\lambda -\D_iR_{jp}\D_p\phi^\lambda
-R_{jp}\D_i\D_p\phi^\lambda\\%
&\quad\,+\big(\NRm{}*\D^2\phi*(\D\phi)^{*2} +
(\partial\NRm{})*(\D\phi)^{*4}\big)_{ij},
\end{aligned}
\end{equation}
see \cite[equation (3.10)]{Muller:diss} for details. With (\ref{3.eq1}) we continue
\begin{align*}
[\D_i\D_j,\Lap]\phi^\lambda -
\big(\dt\Gamma_{ij}^k\big)\D_k\phi^\lambda
&=\big(\Rm{}*\D^2\phi^\lambda\big)_{ij}
-2\alpha\D_i\D_j\phi\D_k\phi\D_k\phi^\lambda\\%
&\quad\,+\big(\NRm{}*\D^2\phi*(\D\phi)^{*2} +
(\partial\NRm{})*(\D\phi)^{*4}\big)_{ij},
\end{align*}
where we used the second Bianchi identity
$(\D_kR_{jpik}+\D_jR_{ip}-\D_pR_{ij})\D_p\phi^\lambda=0$ to cancel
all terms containing derivatives of the curvature of $(M,g)$. Since
the $\D^2\phi$ live in different bundles for different times, we
work again with the covariant time derivative $\D_t$ (and with 
the interpretation of $\D^2\phi$ as a $2$-linear $TN$-valued map 
along $\tilde{\phi}$), as we already did in Section 4, see appendix for
details. At the base point of coordinates satisfying
(\ref{app.eq13}), we find with (\ref{app.eq14}) and the remark
following it
\begin{equation}\label{3.eq11}
\begin{aligned}
\D_t(\D_i\D_j\phi^\lambda) &=\D_i\D_j\dt\phi^\lambda
-\big(\dt\Gamma_{ij}^k\big)\D_k\phi^\lambda +
\NRm{\dt\phi,\D_i\phi}\D_j\phi^\lambda\\%
&= \Lap\D_i\D_j\phi^\lambda +
(\Rm{}*\D^2\phi^\lambda)_{ij} + \alpha \D_i\D_j\phi*\D\phi*\D\phi^\lambda\\%
&\quad\,+\big(\NRm{}*\D^2\phi*(\D\phi)^{*2} +
(\partial\NRm{})*(\D\phi)^{*4}\big)_{ij}.
\end{aligned}
\end{equation}
With $\Lap\abs{\D^2\phi}^2
=2\Lap(\D_i\D_j\phi^\lambda)\D_i\D_j\phi^\lambda
+2\abs{\D^3\phi}^2$, we finally obtain
\begin{equation}\label{3.eq12}
\begin{aligned}
\dt\abs{\D^2\phi}^2 &=(\dt g^{-1})*(\D^2\phi)^{*2}
+2\D_t(\D_i\D_j\phi^\lambda)\D_i\D_j\phi^\lambda\\%
&= \Rm{}*(\D^2\phi)^{*2}+\alpha(\D\phi)^{*2}*(\D^2\phi)^{*2}
+\Lap\abs{\D^2\phi}^2-2\abs{\D^3\phi}^2\\%
&\quad\,+\NRm{}*(\D^2\phi)^{*2}*(\D\phi)^{*2} +
(\partial\NRm{})*(\D^2\phi)*(\D\phi)^{*4}
\end{aligned}
\end{equation}
Since $\abs{\NRm{}}$ and $\abs{\partial \NRm{}}$ are bounded on
compact manifolds $N$, say by a constant $c_1$, this proves the
following proposition.
\begin{prop}\label{3.prop5}
Let $(g(t),\phi(t))_{t\in[0,T)}$ be a solution of\/ $(RH)_\alpha$.
Then the norm of the Hessian of $\phi$ satisfies the estimate
\begin{equation}\label{3.eq13}
\begin{aligned}
\dt\abs{\D^2\phi}^2 &\leq \Lap\abs{\D^2\phi}^2 -2\abs{\D^3\phi}^2
+C\abs{\Rm{}}\abs{\D^2\phi}^2\\%
&\quad + \alpha C\abs{\D\phi}^2\abs{\D^2\phi}^2
+Cc_1\abs{\D\phi}^4\abs{\D^2\phi} + Cc_1
\abs{\D\phi}^2\abs{\D^2\phi}^2
\end{aligned}
\end{equation}
along the flow for some constants $C=C(m)\geq 0$ and $c_1=c_1(N)\geq
0$ depending on the dimension $m$ of $M$ and the curvature of $N$,
respectively. If $N$ is flat, we may choose $c_1=0$.
\end{prop}
\begin{rem}
If we set $\alpha \equiv 0$, Corollary \ref{3.cor4} and Proposition
\ref{3.prop5} yield the formulas for the Ricci-DeTurck flow
$(RH)_0$, in particular (\ref{3.eq9}) reduces to the well-known
evolution inequality
\begin{equation*}
\dt \abs{\Rm{}}^2 \leq \Lap
\abs{\Rm{}}^2-2\abs{\D\Rm{}}^2+C\abs{\Rm{}}^3
\end{equation*}
for the Ricci flow, first derived by Hamilton \cite[Corollary
13.3]{Hamilton:3folds}. Moreover, if $\alpha \equiv 2$ and $N\subseteq\RR$
(and thus $c_0=c_1=0$), the estimates (\ref{3.eq9}) and
(\ref{3.eq13}) reduce to the estimates found by List (cf.
\cite[Lemma 2.15 and 2.16]{List:diss}).
\end{rem}

\subsection{Interior-in-time higher order gradient estimates}
Using the evolution equations for the curvature tensor and the
Hessian of $\phi$, we get evolution equations for higher order
derivatives by induction.
\begin{defn}\label{3.defn6}
To keep the notation short, we define for $k\geq 0$
\begin{equation}\label{3.eq14}
\begin{aligned}
I_k &:=\sum_{i+j=k}\D^i\Rm{}*\D^j\Rm{} +
\alpha\sum_{A_k}(\partial^i\NRm{}+1)*\D^{j_1}\phi*\ldots*\D^{j_\ell}\phi\\%
&\quad\;+\alpha\sum_{B_k}\D^{j_1}\phi*\ldots*
\D^{j_{\ell-1}}\phi*\D^{j_\ell}\Rm{},
\end{aligned}
\end{equation}
where the last two sums are taken over all elements of the index
sets defined by
\begin{align*}
A_k &:= \{(i,j_1,\ldots,j_\ell)\mid 0\leq i\leq k+1,\, 1\leq
j_s\leq k+2\;\forall s \textrm{ and } j_1+\ldots +j_\ell=k+4\},\\%
B_k &:= \{(j_1,\dots,j_\ell)\mid 1\leq j_s< k+2\;\forall s<\ell,\,
0\leq j_\ell\leq k \textrm{ and } j_1+\ldots + j_\ell=k+2\}.
\end{align*}
\end{defn}
\begin{lemma}\label{3.lamma7}
Let $(g(t),\phi(t))_{t\in[0,T)}$ be a solution to the $(RH)_\alpha$
flow. Then for $k\geq 0$, with $I_k$ defined as in (\ref{3.eq14}),
we obtain
\begin{equation}\label{3.eq15}
\dt \D^k \Rm{} = \Lap \D^k\Rm{} + I_k.
\end{equation}
\end{lemma}
\begin{rem}
If $(N,\gamma)=(\RR,\delta)$, all terms in $I_k$ containing
$\partial^i\NRm{}$ vanish, and the result reduces to (a slightly
weaker version of) List's result \cite[Lemma 2.19]{List:diss}. Note
that we do not need all the elements of $A_k$ here, but defining
$A_k$ this way allows us to use the same index set again in
Definition \ref{3.defn8}.
\end{rem}
\begin{proof}
From (\ref{3.eq7}), we see that (\ref{3.eq15}) holds for $k=0$. For
the induction step, assume that (\ref{3.eq15}) holds for some
$k\geq0$ and compute
\begin{equation*}
\dt\D^{k+1}\Rm{}=\dt(\partial\D^k\Rm{}+\Gamma*\D^k\Rm{})
=\D(\Lap\D^k\Rm{}) +\D I_k +\dt\Gamma*\D^k\Rm{}.
\end{equation*}
Since $\D I_k$ is of the form $I_{k+1}$ and also
$\dt\Gamma*\D^k\Rm{}=\big(\D\Rm{}+\alpha\D^2\phi*\D\phi\big)*\D^k\Rm{}$
appears in $I_{k+1}$, it remains to compute the very first term.
With the commutator rule (\ref{app.eq7}), we get
\begin{align*}
\D(\Lap\D^k\Rm{})&=\Lap\D^{k+1}\Rm{}+\D\Rm{}*\D^k\Rm{}+\Rm{}
*\D^{k+1}\Rm{}\\%
&=\Lap\D^{k+1}\Rm{}+I_{k+1}.\qedhere
\end{align*}
\end{proof}
Similar to (\ref{3.eq8}), we obtain
\begin{align*}
\dt\abs{\D^k\Rm{}}^2 &= \big(\dt g^{-1}\big)*\D^k\Rm{}*\D^k\Rm{}
+2\D^k\Rm{}\big(\dt\D^k\Rm{}\big)\\%
&=\Rm{}*(\D^k\Rm{})^{*2}+\alpha\,(\D\phi)^{*2}*(\D^k\Rm{})^{*2}
+2\D^k\Rm{}(\Lap\D^k\Rm{})+\D^k\Rm{}*I_k.
\end{align*}
Hence, using the fact that $\Rm{}*\D^k\Rm{}$ as well as
$\alpha\,(\D\phi)^{*2}*\D^k\Rm{}$ are already contained in $I_k$, we
find
\begin{equation}\label{3.eq16}
\dt\abs{\D^k\Rm{}}^2=\Lap\abs{\D^k\Rm{}}^2
-2\abs{\D^{k+1}\Rm{}}^2+\D^k\Rm{}*I_k.
\end{equation}
\begin{defn}\label{3.defn8}
To compute the higher order derivatives of $\phi$, we define
\begin{equation}\label{3.eq17}
\begin{aligned}
J_k &:=\sum_{i+j=k}\D^i\Rm{}*\D^{j+2}\phi +
\sum_{A_k}(\partial^i\NRm{}+1)*\D^{j_1}\phi*\ldots*\D^{j_\ell}\phi\\%
&\quad\;+\alpha\sum_{B_k}\D^{j_1}\phi*\ldots*
\D^{j_{\ell-1}}\phi*\D^{j_\ell+2}\phi,
\end{aligned}
\end{equation}
with $A_k$ and $B_k$ defined as in Definition \ref{3.defn6}.
\end{defn}
\begin{lemma}\label{3.lemma9}
Let $(g(t),\phi(t))_{t\in[0,T)}$ be a solution to the $(RH)_\alpha$
flow. Then for $k\geq 0$, with $J_k$ defined as in (\ref{3.eq17}),
we have
\begin{equation}\label{3.eq18}
\D_t(\D^{k+2}\phi) = \Lap \D^{k+2} \phi + J_k.
\end{equation}
\end{lemma}
\begin{proof}
For $k=0$, the statement holds by (\ref{3.eq11}). For the induction
step, we use again the interpretation of $\D^k\phi$ as a $k$-linear
$TN$-valued map along $\tilde{\phi}$ and compute analogously to
(\ref{app.eq14}) and the remark following it
\begin{align*}
\D_t(\D^{k+3}\phi) &=\D\D_t(\D^{k+2}\phi) +
\dt\Gamma*\D^{k+2}\phi + \NRm{\dt\phi,\D\phi}\D^{k+2}\phi\\
&= \D(\Lap\D^{k+2}\phi)+\D J_k +\dt\Gamma*\D^{k+2}\phi +
\NRm{}*\D^2\phi*\D\phi*\D^{k+2}\phi.
\end{align*}
Again, we only have to look at the first term, since $\D J_k$,
$\NRm{}*\D^2\phi*\D\phi*\D^{k+2}\phi$ and $\dt\Gamma*\D^{k+2}\phi =
\big(\D\Rm{}+\alpha\D^2\phi*\D\phi\big)
*\D^{k+2}\phi$ are obviously of the form $J_{k+1}$. With a higher
order analog to (\ref{app.eq10}), we obtain
\begin{align*}
\D(\Lap\D^{k+2}\phi)&=\D_p\D\D_p\D^{k+2}\phi +\Rm{}*\D^{k+3}\phi +
\NRm{}*\D^{k+3}\phi*\D\phi*\D\phi\\%
&=\D_p\D_p\D^{k+3}\phi + \D\Rm{}*\D^{k+2}\phi+ \Rm{}*\D^{k+3}\phi\\%
&\quad\,+(\partial \NRm{})*\D^{k+2}\phi*(\D\phi)^{*3} +
\NRm{}*\D^{k+3}\phi*(\D\phi)^{*2}\\%
&\quad\,+\NRm{}*\D^{k+3}\phi*\D^2\phi*\D\phi\\%
&=\Lap\D^{k+3}\phi+J_{k+1},
\end{align*}
and the claim follows.
\end{proof}
As in (\ref{3.eq12}), we compute
\begin{align*}
\dt\abs{\D^{k+2}\phi}^2&=\big(\dt g^{-1}\big)
*\D^{k+2}\phi*\D^{k+2}\phi+2\D^{k+2}\phi^\lambda
\D_t\big(\D^{k+2}\phi^\lambda\big)\\
&= \Rm{}*(\D^{k+2}\phi)^{*2}+\alpha\,
(\D\phi)^{*2}*(\D^{k+2}\phi)^{*2}\\
&\quad+2\D^{k+2}\phi^\lambda(\Lap\D^{k+2}\phi^\lambda) +
\D^{k+2}\phi*J_k\\
&= 2\D^{k+2}\phi^\lambda(\Lap\D^{k+2}\phi^\lambda)+\D^{k+2}\phi*J_k.
\end{align*}
With $\Lap\abs{\D^{k+2}\phi}^2=2\D^{k+2}\phi^\lambda(\Lap\D^{k+2}\phi^\lambda)
+2\abs{\D^{k+3}\phi}^2$, we finally find
\begin{equation}\label{3.eq19}
\dt\abs{\D^{k+2}\phi}^2=\Lap\abs{\D^{k+2}\phi}^2-2\abs{\D^{k+3}\phi}^2
+ \D^{k+2}\phi*J_k.
\end{equation}
The next trick is to combine the two equations (\ref{3.eq16}) and
(\ref{3.eq19}) to a single equation. Remember that we already used a
similar idea in Section 4, where we combined the evolution
equations of $\Rc{}$ and $\D\phi\otimes\D\phi$ (respectively $R$ and
$\abs{\D\phi}^2$) to a single equation for a combined quantity
$S_{ij}$ (respectively $S$), which was much more convenient to deal
with. Here, we define the ``vector''
\begin{equation}\label{3.eq20}
\sT=(\Rm{},\D^2\phi) \in \Gamma\big((T^*M)^{\otimes 4}\big) \times
\Gamma\big((T^*M)^{\otimes 2}\otimes\phi^*TN\big)
\end{equation}
with norm $\abs{\sT}^2 = \abs{\Rm{}}^2 + \abs{\D^2\phi}^2$ and
derivatives $\D^k\sT=(\D^k\Rm{},\D^{k+2}\phi)$. Combining the
evolution equations (\ref{3.eq16}) and (\ref{3.eq19}), we get
\begin{equation}\label{3.eq21}
\dt\abs{\D^k\sT}^2 = \Lap\abs{\D^k\sT}^2-2\abs{\D^{k+1}\sT}^2 + L_k,
\end{equation}
where $L_k := \D^k\Rm{}*I_k + \D^{k+2}\phi*J_k$. We can now apply
Bernstein's ideas \cite{Bernstein:estimates} to obtain
interior-in-time estimates for all derivatives $\abs{\D^k\sT}^2$ via
an induction argument. For the Ricci flow, this was independently
done by Bando \cite{Bando:estimates} and Shi \cite{Shi:complete}. 
\begin{thm}\label{3.thm10}
Let $(g(t),\phi(t))_{t\in[0,T)}$ solve $(RH)_\alpha$ with
non-increasing $\alpha(t)\in[\underaccent{\bar}\alpha,\bar\alpha]$,
$0<\underaccent{\bar}\alpha\leq\bar\alpha<\infty$ and $T<\infty$.
Let the Riemannian curvature tensor of $M$ be uniformly bounded
along the flow, $\abs{\Rm{}}\leq R_0$. Then there exists a constant
$K=K(\underaccent{\bar}\alpha,\bar\alpha,R_0,T,m,N)<\infty$ such
that the following two estimates hold
\begin{align}
\abs{\D\phi}^2 &\leq \tfrac{K}{t}, \quad \forall (x,t)\in
M\times(0,T),\label{3.eq22}\\
\abs{\sT}^2 &=\abs{\Rm{}}^2+\abs{\D^2\phi}^2 \leq \tfrac{K^2}{t^2},
\quad \forall (x,t)\in M\times(0,T).\label{3.eq23}
\end{align}
Moreover, there exist constants $C_k$ depending on $k$,
$\bar\alpha$, $m$ and $N$, such that
\begin{equation}\label{3.eq24}
\abs{\D^k\sT}^2 = \abs{\D^k\Rm{}}^2+\abs{\D^{k+2}\phi}^2 \leq C_k
\big(\tfrac{K}{t}\big)^{k+2}, \quad\forall (x,t)\in M\times(0,T).
\end{equation}
\end{thm}
\begin{proof}
Since the method of proof is quite standard, we only give a brief sketch of the argument and
refer to the authors thesis \cite[Theorem 3.10]{Muller:diss} for more details. Setting $1\leq K_1 :=
\max\big\{\frac{2m^2R_0T+m}{2\underaccent{\bar}\alpha}, R_0T,
1\big\}<\infty$, we obtain
\begin{equation*}
\abs{\D\phi}^2\leq \frac{m^2R_0}{\underaccent{\bar}\alpha}
+\frac{m}{2\underaccent{\bar}\alpha t} \leq \frac{K_1}{t} \quad
\textrm{ and }\quad \abs{\Rm{}}\leq \frac{K_1}{t}, \quad \forall
(x,t)\in M\times(0,T)
\end{equation*}
from Corollary \ref{2.cor9}. In the following, $C$ denotes a
constant depending on $K_1$, $\bar\alpha$, $m$ and the geometry
of $N$, possibly changing from line to line. With the estimates for
$\abs{\Rm{}}$ and $\abs{\D\phi}^2$, and using $\abs{\D^2\phi}\leq
\frac{1}{t} + t\abs{\D^2\phi}^2$, we obtain for $f(x,t):=
t^2\abs{\D^2\phi}^2\big(8K_1 + t\abs{\D\phi}^2\big)$
\begin{align*}
\big(\dt-\Lap\big)f
&\leq -2t^2\abs{\D^3\phi}^2\big(8K_1+t\abs{\D\phi}^2\big) +
\tfrac{C}{t}f + \tfrac{C}{t}\cdot 9K_1 -2t^3\abs{\D^2\phi}^4 +
\tfrac{C}{t}f\\
&\quad\, + 8t^3\abs{\D^3\phi}\abs{\D^2\phi}
\cdot\abs{\D^2\phi}\abs{\D\phi}
\end{align*}
on $M\times(0,T)$. The last term can be absorbed by the two negative terms,
\begin{align*}
8t^3\abs{\D^3\phi}\abs{\D^2\phi}^2\abs{\D\phi}
&\leq \tfrac{1}{2}(8K_1)\big(4t^2\abs{\D^3\phi}^2\big) +
\tfrac{1}{2}(8K_1)^{-1}\big(16t^4\abs{\D^2\phi}^4\abs{\D\phi}^2\big)\\
&=2t^2\abs{\D^3\phi}^2\cdot 8K_1 +
\tfrac{8t\abs{\D\phi}^2}{8K_1}\cdot t^3\abs{\D^2\phi}^4\\
&\leq 2t^2\abs{\D^3\phi}^2\big(8K_1+t\abs{\D\phi}^2\big) +
t^3\abs{\D^2\phi}^4.
\end{align*}
Here, we used $\tfrac{8t\abs{\D\phi}^2}{8K_1}\leq 1$ which motivates
our choice of the constant $8K_1$ in the definition of $f$. From
$\big(\dt-\Lap\big)f \leq\tfrac{C}{t}f+\tfrac{C}{t}-t^3\abs{\D^2\phi}^4 \leq
\tfrac{1}{(9K_1)^2t}\big(Cf + C-f^2\big)$, we conclude, using $f(\cdot,0)=0$ and the maximum principle, that $-f^2+Cf+C\geq 0$. Equivalently, $f \leq D:=\frac{1}{2}\big(C + \sqrt{C^2+4C}\big)$ on $M\times[0,T)$. For
positive $t$, this implies
\begin{equation}\label{3.eq28}
\abs{\D^2\phi}^2 = \frac{f}{t^2(8K_1+t\abs{\D\phi}^2)} \leq
\frac{D}{8K_1t^2}\leq \Big(\frac{K_2}{t}\Big)^2,
\end{equation}
where $K_2:= \sqrt{D/8K_1}<\infty$. Setting $K:=K_1+K_2$, we get
(\ref{3.eq22}) and (\ref{3.eq23}). Using a similar argument, one can then prove (\ref{3.eq24}) inductively. The crucial estimates are $L_k \leq C\big(\tfrac{K}{t}\big)^{k+3}$ and
\begin{equation*}
L_{k+1} \leq C\big(\tfrac{K}{t}\big)^{k+4}+C\tfrac{K}{t}\abs{\D^{k+1}\sT}^2,
\end{equation*}
where $C$ now denotes a constant depending only on $\bar\alpha$, $m$, $N$ and $k$
(but not on $K$ or $T$). Defining $h(x,t):=t^{k+3}\abs{\D^{k+1}\sT}^2\big(\lambda +
t^{k+2}\abs{\D^k\sT}^2\big)$ with $\lambda = 8C_kK^{k+2}$, these estimates give
\begin{align*}
\big(\dt-\Lap\big)h &\leq -2t^{k+3}\abs{\D^{k+2}\sT}^2\big(\lambda + t^{k+2}
\abs{\D^k\sT}^2\big)+\tfrac{CK}{t}h+\tfrac{C}{t}K^{k+4}\big(\lambda
+ t^{k+2}\abs{\D^k\sT}^2\big)\\
&\quad\,-2t^{2k+5}\abs{\D^{k+1}\sT}^4+\tfrac{C}{t}h+\tfrac{C}{t}K^{k+3}
t^{k+3}\abs{\D^{k+1}\sT}^2\\
&\quad\,+8t^{2k+5}\abs{\D^{k+2}\sT}\abs{\D^{k+1}\sT}\cdot
\abs{\D^{k+1}\sT}\abs{\D^k\sT}.
\end{align*}
Using $K\geq 1$, the inductive assumption and Cauchy-Schwarz, we rewrite this as
\begin{align*}
\big(\dt-\Lap\big)h &\leq -2t^{k+3}\abs{\D^{k+2}\sT}^2\big(\lambda +
t^{k+2} \abs{\D^k\sT}^2\big) -\tfrac{3}{2}t^{2k+5}
\abs{\D^{k+1}\sT}^4\\
&\quad\,+\tfrac{CK}{t}h+\tfrac{C}{t}K^{2k+6}
+8t^{2k+5}\abs{\D^{k+2}\sT}\abs{\D^{k+1}\sT}^2\abs{\D^k\sT}.
\end{align*}
Again, the last term can be absorbed by the negative terms
\begin{align*}
8t^{2k+5}\abs{\D^{k+2}\sT}\abs{\D^{k+1}\sT}^2\abs{\D^k\sT}
&\leq\tfrac{1}{2}\lambda\big(4t^{k+3}\abs{\D^{k+2}\sT}^2\big) +
\tfrac{1}{2}\lambda^{-1}\big(16t^{3k+7}\abs{\D^{k+1}\sT}^4
\abs{\D^k\sT}^2\big)\\
&=2t^{k+3}\abs{\D^{k+2}\sT}^2\cdot \lambda +
\tfrac{8t^{k+2}\abs{\D^k\sT}^2}{\lambda}\cdot
t^{2k+5}\abs{\D^{k+1}\sT}^4\\
&\leq 2t^{k+3}\abs{\D^{k+2}\sT}^2\big(\lambda
+t^{k+2}\abs{\D^k\sT}^2\big)+t^{2k+5}\abs{\D^{k+1}\sT}^4,
\end{align*}
which explains our choice of $\lambda$. A maximum principle argument like the one for $f$ above then yields $-h^2+CK^{2k+5}h+CK^{4k+10} \geq 0$, i.e.~$h \leq\tfrac{1}{2}(C+\sqrt{C^2+4C})K^{2k+5}=:DK^{2k+5}$ on $M\times[0,T)$. For $t>0$,
\begin{equation*}
\abs{\D^{k+1}\sT}^2 =
\frac{h}{t^{k+3}(\lambda+t^{k+2}\abs{\D^k\sT}^2)}\leq
\frac{DK^{2k+5}}{t^{k+3}8C_kK^{k+2}} =
C_{k+1}\Big(\frac{K}{t}\Big)^{k+3},
\end{equation*}
where $C_{k+1}:=D/(8C_k)$. This proves the induction step and hence
the theorem.
\end{proof}
In the following corollary, we state a local version of the gradient estimates.
The setting is made in such a way to perfectly fit the proof of the non-collapsing 
result in Section 8.
\begin{cor}\label{app.prop5}
Let $(g(t),\phi(t))_{t\in[0,T)}$ solve $(RH)_\alpha$ with
non-increasing $\alpha(t)\in[\underaccent{\bar}\alpha,\bar\alpha]$,
$0<\underaccent{\bar}\alpha\leq\bar\alpha<\infty$ and $T'<T<\infty$.
Let $B:=B_{g(T')}(x,r)$ be a ball around $x$ with radius $r$,
measured with respect to the metric at time $T'$. Assume that
$\abs{\Rm{}}\leq R_0$ on the set $B\times[0,T')$. Then there exist
constants $K=K(\underaccent{\bar}\alpha,\bar\alpha,
R_0,T,m,N)<\infty$ and $C_k=C_k(k,\bar\alpha,m,N)$ for $k\in\NN$,
$C_0=1$, such that the following estimates hold for $k\geq0$
\begin{align}
\abs{\D\phi}^2 &\leq
\tfrac{K}{t}\quad\text{and}\quad\abs{\Rm{}}\leq\tfrac{K}{t},
\quad\forall (x,t)\in B^{1/2}\times(0,T'),\label{app.eq26}\\
\abs{\D^k\sT}^2 &= \abs{\D^k\Rm{}}^2+\abs{\D^{k+2}\phi}^2 \leq C_k
\big(\tfrac{K}{t}\big)^{k+2}, \quad\forall(x,t)\in
B^{1/2}\times(0,T'),\label{app.eq27}
\end{align}
where $B^{1/2}:=B_{g(T')}(x,r/2)$ is the ball of half the radius and
the same center as $B$.
\end{cor}
\begin{proof}
The statement (\ref{app.eq26}) follows exactly as in Theorem
\ref{3.thm10}. The induction step is carried out using a cut-off function 
to ensure that the maxima are attained in the interior of the set $B$. 
More details can be found in the authors thesis 
\cite[Proposition A.5]{Muller:diss}.
\end{proof}

\subsection{Long-time existence}
This subsection follows Section 6.7 about long-time existence for
the Ricci flow from Chow and Knopf's book \cite{RF:intro}. We first
need a technical lemma.
\begin{lemma}\label{3.lemma11}
Let $(g(t),\phi(t))_{t\in[0,T)}$ solve $(RH)_\alpha$ with a
non-increasing $\alpha(t)\in[\underaccent{\bar}\alpha,\bar\alpha]$,
$0<\underaccent{\bar}\alpha\leq\bar\alpha<\infty$ and $T<\infty$.
Let the Riemannian curvature tensor of $M$ be uniformly bounded
along the flow, $\abs{\Rm{}}\leq R_0$, and fix a background metric
$\tilde{g}$. Then for each $k\geq 0$ there exists a constant $C_k$
depending on $k$, $m$, $N$, $T$, $\underaccent{\bar}\alpha$,
$\bar\alpha$, $R_0$ and the initial data $(g(0),\phi(0))$ such that
\begin{equation}\label{3.eq32}
\abs{\tilde{\D}^kg(x,t)}_{\tilde{g}}^2 +
\abs{\tilde{\D}^k\Rm{x,t}}_{\tilde{g}}^2 +
\abs{\tilde{\D}^k\phi(x,t)}_{\tilde{g}}^2 \leq C_k
\end{equation}
for all $(x,t)\in M\times[0,T)$. Here, $\tilde{\D}=
\!\phantom{l}^{\tilde{g}}\D$ denotes the Levi-Civita connection with
respect to the background metric $\tilde{g}$.
\end{lemma}
\begin{proof}
With Theorem \ref{3.eq10}, the proof becomes a straight forward
computation, and we therefore only give a sketch. In \cite[Section
6.7]{RF:intro}, all the details are carried out in the case of the
Ricci flow and they can easily be adopted to our flow. Since $M$ is
closed, there exists a finite atlas for which we have uniform bounds
on the derivatives of the local charts. Working in such a chart
$\psi:U\to\RR^m$, it suffices to derive the desired estimates for
the Euclidean metric $\delta$ in $U$ and the ordinary derivatives,
since $\tilde{g}$ and $\tilde{\D}$ are fixed. In particular, we can
interpret $\Gamma$ as a tensor, namely
$\Gamma=\Gamma-\!\phantom{l}^\delta\Gamma$. On the compact interval
$[0,T/2]$, all the derivatives $\abs{\D^k\phi}^2_g$ and
$\abs{\D^k\Rm{}}^2_g$ are uniformly bounded. On the interval
$[T/2,T)$, Theorem \ref{3.thm10} above yields uniform bounds for
these derivatives. Hence
\begin{equation}\label{3.eq33}
\abs{\D^k\phi}^2_g + \abs{\D^k\Rm{}}^2_g \leq \bar{C}_k
\end{equation}
for some $\bar{C}_k<\infty$. In particular,
$\sS=\Rc{}-\alpha\D\phi\otimes\D\phi$ is uniformly bounded on
$[0,T)$. From \cite[Lemma 6.49]{RF:intro}, we infer that all $g(t)$
are uniformly equivalent on $[0,T)$, and thus for some constant $C$
\begin{equation}\label{3.eq34}
C^{-1}\delta\leq g(x,t)\leq C\delta, \quad \forall (x,t)\in
U\times[0,T).
\end{equation}
With $\dt(\partial g)=\partial(\dt g)= -2\partial\sS = -2(\D\sS +
\Gamma*\sS)$, we compute
\begin{equation}\label{3.eq35}
\abs{\dt\partial g}_{\delta}\leq C\abs{\dt\partial g}\leq
C\abs{\D\sS}+C\abs{\Gamma}\abs{\sS}.
\end{equation}
Thus, we have
\begin{equation*}
\abs{\dt\Gamma}\leq
C\abs{\D\Rc{}}+2\bar{\alpha}\abs{\D\phi}\abs{\D^2\phi}
\end{equation*}
which yields a bound for $\abs{\Gamma}$ by integration. Together
with the bounds for $\abs{\sS}$ and $\abs{\D\sS}$ that we obtain
from (\ref{3.eq33}), we conclude that $\abs{\dt\partial g}_{\delta}$
is uniformly bounded, and hence -- again by integration --
$\abs{\partial g}_{\delta}$ is uniformly bounded on $U\times[0,T)$.
Finally, using a partition of unity for the chosen atlas, we obtain
a uniform bound for $\abs{\tilde{\D}g}_{\tilde{g}}$ on $M\times
[0,T)$. A short computation as for (\ref{3.eq35}) yields
\begin{equation}\label{3.eq36}
\abs{\dt\tilde{\D}^kg}_{\tilde{g}} \leq C \abs{\dt\tilde{\D}^kg}
\leq \sum_{i=0}^k c_i\abs{\Gamma}^i\abs{\D^{k-i}\sS} +
\sum_{i=1}^{k-1}c_i'\abs{\partial^i\Gamma}\abs{\tilde{\D}^{k-1-i}\sS},
\end{equation}
where the constants $c_i$, $c_i'$ only depend on $m$ and $k$. From
this formula, we inductively obtain the desired bounds for
$\abs{\tilde{\D}^kg}_{\tilde{g}}$. Similarly, the estimates for
$\abs{\tilde{\D}^k\Rm{x,t}}_{\tilde{g}}$ and
$\abs{\tilde{\D}^k\phi(x,t)}_{\tilde{g}}$ are obtained from
(\ref{3.eq33}) with a transformation analogous to (\ref{3.eq36}).
The lemma then follows by plugging everything together.
\end{proof}
Finally, we obtain our desired criterion for long-time existence.
\begin{thm}\label{3.thm12}
Let $(g(t),\phi(t))_{t\in[0,T)}$ solve $(RH)_\alpha$ with
non-increasing $\alpha(t)\in[\underaccent{\bar}\alpha,\bar\alpha]$,
$0<\underaccent{\bar}\alpha\leq\bar\alpha<\infty$ and $T<\infty$.
Suppose that $T<\infty$ is maximally chosen, i.e.~the solution
cannot be extended beyond $T$ in a smooth way. Then the curvature of
$(M,g(t))$ has to become unbounded for $t\nearrow T$ in the sense
that
\begin{equation}\label{3.eq37}
\limsup_{t\nearrow T}\Big(\max_{x\in M}\abs{\Rm{x,t}}^2\Big) =
\infty.
\end{equation}
\end{thm}
\begin{proof}
The proof is by contradiction. Suppose that the curvature stays
bounded on $[0,T)$, say $\abs{\Rm{}}\leq R_0$. For any point $x\in
M$ and vector $X\in T_xM$, define $g(x,T)(X,X):=\lim_{t\to
T}g(x,t)(X,X)$. We estimate
\begin{equation*}
\abs{g(x,T)(X,X)-g(x,t)(X,X)}\leq \int_t^T 2\abs{\sS(x,\tau)(X,X)}d\tau 
\leq C\abs{X}^2(T-t),
\end{equation*}
where we used again the fact that $\sS$ is uniformly bounded on
$M\times[0,T)$. This shows that the limit $g(x,T)(X,X)$ is well
defined and continuous in $x$. Hence, we obtain a continuous limit
$g(\cdot,T)\in \Gamma(\Sym^2(T^*M))$ by polarization. From
\cite[Lemma 6.49]{RF:intro}, all metrics $g(\cdot,t)$ are uniformly
equivalent -- as in (\ref{3.eq34}) -- which implies that this limit
must be a (continuous) Riemannian metric. Moreover, we define
$\phi(x,T):=\lim_{t\to T}\phi(x,t)$, where we use again the
embedding $e_N:N\hookrightarrow\RR^d$ to interpret $\phi$ as a map
into $\RR^d$. We estimate
\begin{equation*}
\abs{\phi(x,T)-\phi(x,t)}\leq \int_t^T \abs{\dt\phi(x,\tau)}d\tau
\leq C(T-t),
\end{equation*}
since the bound on $\abs{\D^2\phi}$ yields a bound on
$\abs{\dt\phi}=\abs{\tau_g\phi}$. This implies that $\phi(\cdot,T)$
is well defined and continuous in $x$. The uniform bounds
(\ref{3.eq32}) from Lemma \ref{3.lemma11} then also hold for the
limit $(g(T),\phi(T))$ and hence $g(T)$ and $\phi(T)$ are smooth.
Indeed, for an arbitrary background metric $\tilde{g}$, we have
\begin{equation*}
\abs{\tilde{\D}^kg(T)-\tilde{\D}^kg(t)}_{\tilde{g}} \leq \int_t^T
\abs{\dt\tilde{\D}^kg(\tau)}_{\tilde{g}}d\tau \leq C(T-t),
\end{equation*}
which follows from the uniform bound for
$\abs{\dt\tilde{\D}g}_{\tilde{g}}$ that we have derived in the lemma
above. This means, the convergence $g(t)\to g(T)$ is smooth. With
\begin{equation*}
\abs{\tilde{\D}^k\phi(T)-\tilde{\D}^k\phi(t)}_{\tilde{g}} \leq
\int_t^T \abs{\dt\tilde{\D}^k\phi(\tau)}_{\tilde{g}}d\tau \leq
C(T-t)
\end{equation*}
we see that also $\phi(t)\to\phi(T)$ uniformly in any $C^k$-norm.
Finally, restarting the flow with $(g(T),\phi(T))$ as new initial
data, we obtain a solution $(g(t),\phi(t))_{t\in[T,T+\eps)}$ by the
short-time existence result from Chapter 2. This yields an extension
of our solution beyond time $T$ which is smooth in space for each
time. From the flow equations and the uniform bounds on
$\abs{\D^k\Rm{}}$ as well as $\abs{\D^k\phi}$, the time derivatives
(and hence also the mixed derivatives) are smooth too, in particular
near $t=T$. This means, the extension of the flow is smooth in space
\emph{and} time, contradicting the maximality of $T$.
\end{proof}

\section{Monotonicity formula and no breathers theorem}
The entropy functional $\sW_\alpha$ introduced in this section is
the analogue of Perelman's shrinker entropy for the Ricci flow from
\cite[Section 3]{Perelman:entropy}. It is obtained from the energy
functional $\sF_\alpha$ from (\ref{1.eq1}), by
introducing a positive scale factor $\tau$ (later interpreted as a
backwards time) and some correction terms. For detailed explanations
of Perelman's result, we again refer to Chow et al. \cite[Chapter
6]{RF:TAI} and M\"{u}ller \cite[Chapter 3]{Muller:Harnack}. Moreover for
the special case $N\subseteq\RR$, the entropy functional $\sW_\alpha$ 
can be found in List's dissertation \cite{List:diss}.

\subsection{The entropy functional and its first variation}
Let again $g=g_{ij} \in \Gamma\big(\Sym^2_{+}(T^*M)\big)$, $\phi \in
C^\infty(M,N)$, $f:M\to\RR$ and $\tau >0$. For a time-independent
coupling constant $\alpha(t) \equiv \alpha > 0$, we set
\begin{equation}\label{5.eq1}
\sW_\alpha(g,\phi,f,\tau) := \int_M \Big( \tau \big(R_g + \abs{\D
f}^2_g - \alpha \abs{\D\phi}^2_g\big) +
f-m\Big)(4\pi\tau)^{-m/2}e^{-f}dV_g.
\end{equation}
As in Section 3, we take variations $g^\eps=g+\eps h$, $f^\eps=f+\eps \ell$,
$\phi^\eps=\pi_N(\phi + \eps \vartheta)$,
such that $\delta g = h$, $\delta f = \ell$ and $\delta \phi =
\vartheta$. Additionally, set $\tau^\eps = \tau + \eps\sigma$ for some $\sigma
\in \RR$, i.e.~$\delta \tau = \sigma$. The variation
\begin{equation*}
\delta\sW_\alpha :=
\delta\sW_{\alpha,g,\phi,f,\tau}(h,\vartheta,\ell,\sigma):=
\frac{d}{d\eps}\Big|_{\eps=0} \sW_\alpha(g+\eps h, \pi_N(\phi +
\eps\vartheta), f+\eps\ell, \tau + \eps \sigma)
\end{equation*}
is easiest computed using the variation of $\sF_\alpha$ and
\begin{equation}\label{5.eq2}
\sW_\alpha(g,\phi,f,\tau) = (4\pi\tau)^{-m/2}\Big(
\tau\,\sF_\alpha(g,\phi,f) + \int_M (f-m)e^{-f}dV \Big).
\end{equation}
Using
\begin{equation*}
\delta \int_M (f-m)e^{-f}dV = \int_M \Big(\ell +
(f-m)\big(\tfrac{1}{2}\tr_g h -\ell\big)\Big)e^{-f}dV
\end{equation*}
and equation (\ref{1.eq2}) for the variation of
$\sF_\alpha(g,\phi,f)$, we get from (\ref{5.eq2})
\begin{align*}
\delta\sW_\alpha &= \int_M -\tau h_{ij}\big(R_{ij} + \D_i\D_j
f-\alpha \D_i \phi \D_j
\phi\big)d\mu\\%
&\quad+ \int_M \tau \big(\tfrac{1}{2}\tr_g h - \ell\big)\big(2\Lap f
- \abs{\D f}^2 + R - \alpha \abs{\D \phi}^2 +
\tfrac{f-m}{\tau}\big)d\mu\\%
&\quad+ \int_M \big(\ell + \sigma(1-\tfrac{m}{2})(R+\abs{\D
f}^2-\alpha\abs{\D\phi}^2)-\tfrac{m\sigma}{2\tau}(f-m)\big)
d\mu\\%
&\quad+ \int_M 2\tau\alpha\vartheta\big(\tau_g \phi -
\Scal{\D\phi,\D f}\big)d\mu,
\end{align*}
where we used the abbreviation $d\mu := (4\pi\tau)^{-m/2}e^{-f}dV$.
Rearranging the terms, writing
\begin{equation*}
\ell = (-\tau h_{ij} +\sigma g_{ij})\big(\tfrac{-1}{2\tau}
g_{ij}\big) + \tau \big(\tfrac{1}{2}\tr_g h -\ell -
\tfrac{m\sigma}{2\tau}\big)\big(\tfrac{-1}{\tau}\big),
\end{equation*}
and using $\int_M (\Lap f-\abs{\D f}^2)d\mu =
-(4\pi\tau)^{-m/2}\int_M \Lap(e^{-f})dV = 0$, we get
\begin{align*}
\delta\sW_\alpha &= \int_M (-\tau h_{ij}+\sigma g_{ij})\big(R_{ij} +
\D_i\D_j f-\alpha \D_i \phi \D_j \phi-
\tfrac{1}{2\tau}g_{ij}\big)d\mu\\%
&\quad+ \int_M \tau \big(\tfrac{1}{2}\tr_g h -
\ell-\tfrac{m\sigma}{2\tau}\big)\big(2\Lap f - \abs{\D f}^2 + R -
\alpha \abs{\D \phi}^2 + \tfrac{f-m-1}{\tau}\big)d\mu\\%
&\quad+ \int_M 2\tau\alpha\vartheta\big(\tau_g \phi -
\Scal{\D\phi,\D f}\big)d\mu.
\end{align*}

\subsection{Fixing the background measure}
Similar to Section 3, we now fix the measure $d\mu =
(4\pi\tau)^{-m/2}e^{-f}dV$. This means, we have $f=
-\log((4\pi\tau)^{m/2}\frac{d\mu}{dV})$ and from $0=\delta d\mu =
(\frac{1}{2}\tr_g h -\ell -\frac{m\sigma}{2\tau})d\mu$, we deduce
$\ell = \frac{1}{2}\tr_g h -\frac{m\sigma}{2\tau}$. Moreover, we
require the variation of $\tau$ to satisfy $\delta\tau = \sigma=-1$.
This allows us to interpret $\tau$ as backwards time later. We then
write
\begin{equation}\label{5.eq4}
\sW_\alpha^\mu(g,\phi,\tau) := \sW_\alpha\big(g,\phi,
-\log\big((4\pi\tau)^{m/2}\tfrac{d\mu}{dV}\big),\tau\big)
\end{equation}
and
\begin{equation*}
\delta \sW^\mu_{\alpha,g,\phi,\tau}(h,\vartheta) := \delta
\sW_{\alpha,g,\phi,-\log((4\pi\tau)^{m/2}\frac{d\mu}{dV}),\tau}
\left(h,\vartheta, \tfrac{1}{2}\tr_g h + \tfrac{m}{2\tau},-1\right).
\end{equation*}
The variation formula above reduces to
\begin{equation}\label{5.eq5}
\begin{aligned}
\delta \sW^\mu_{\alpha,g,\phi,\tau}(h,\vartheta) &= \int_M (-\tau
h_{ij}- g_{ij})\big(R_{ij} + \D_i\D_j f-\alpha \D_i \phi \D_j \phi-
\tfrac{1}{2\tau}g_{ij}\big)d\mu\\%
&\quad+ \int_M 2\tau\alpha\vartheta\big(\tau_g \phi -
\Scal{\D\phi,\D f}\big)d\mu,
\end{aligned}
\end{equation}
which is monotone under the gradient-like system of evolution
equations given by
\begin{equation}\label{5.eq6}
\left\{\begin{aligned}\dt g_{ij} &= -2(R_{ij}+\D_i\D_j f -\alpha
\D_i
\phi \D_j\phi),\\%
\dt \phi &= \tau_g \phi - \Scal{\D\phi,\D f},\\%
\dt f &= -R -\Lap f + \alpha \abs{\D\phi}^2 +
\tfrac{m}{2\tau},\\%
\dt \tau &= -1.\end{aligned}\right.
\end{equation}%
As in Section 3, pulling back the solutions of (\ref{5.eq6}) with
the family of diffeomorphisms generated by $\D f$, we get a solution
of
\begin{equation}\label{5.eq7}
\left\{\begin{aligned}\dt g &= -2\Rc{} + 2\alpha \D\phi \otimes \D\phi,\\%
\dt \phi &= \tau_g \phi,\\%
\dt \tau &= -1,\\%
0 &= \Box^*((4\pi\tau)^{-m/2}e^{-f}). \end{aligned}\right.
\end{equation}
Since $\sW_\alpha$ is diffeomorphism invariant, we find the analogue
to Proposition \ref{1.prop1}.
\begin{prop}\label{5.prop1}
Let $(g(t),\phi(t))_{t \in [0,T)}$ be a solution of\/ $(RH)_\alpha$
with $\alpha(t)\equiv \alpha>0$, $\tau$ a backwards time with $\dt
\tau = -1$ and $(4\pi\tau)^{-m/2}e^{-f}$ a solution of the adjoint
heat equation under the flow. Then the entropy functional
$\sW_\alpha(g,\phi,f,\tau)$ is non-decreasing with
\begin{equation}\label{5.eq8}
\begin{aligned}
\frac{d}{dt} \sW_\alpha &= \int_M 2\tau\Abs{\Rc{}-\alpha \D\phi
\otimes \D\phi+\Hess(f)-\tfrac{g}{2\tau}}^2
(4\pi\tau)^{-m/2}e^{-f}dV\\%
&\quad +\int_M 2\tau\alpha \Abs{\tau_g \phi-\Scal{\D\phi,\D f}}^2
(4\pi\tau)^{-m/2}e^{-f}dV.
\end{aligned}
\end{equation}
\end{prop}
\begin{rem}
As seen in Corollary \ref{1.cor2} for the energy $\sF_\alpha$, the
monotonicity of the entropy $\sW_\alpha$ also holds true for
non-increasing positive coupling functions $\alpha(t)$
instead of a constant $\alpha>0$.
\end{rem}

\subsection{Minimizing over all probability measures}
Similar to $\lambda_\alpha(g,\phi)$ defined in (\ref{1.eq13}), we
set
\begin{equation}\label{5.eq9}
\begin{aligned}
\mu_\alpha(g,\phi,\tau) &:= \inf \big\{\sW_\alpha^\mu(g,\phi,\tau)
\mid \mu(M)=1 \big\}\\
&\phantom{:}= \inf \left\{\sW_\alpha(g,\phi,f,\tau) \; \bigg| \;
\int_M (4\pi\tau)^{-m/2}e^{-f}dV =1 \right\}.
\end{aligned}
\end{equation}
Our goal is again to show that the infimum is always achieved. Note
that for $\tilde{g}=\tau g$ we have $R_{\tilde{g}} =
\frac{1}{\tau}R_g$, $\abs{\D f}^2_{\tilde{g}} =
\frac{1}{\tau}\abs{\D f}^2_g$, $\abs{\D \phi}^2_{\tilde{g}} =
\frac{1}{\tau}\abs{\D \phi}^2_g$, $dV_{\tilde{g}}=\tau^{m/2}dV_g$
and thus
\begin{equation*}
\mu_\alpha(\tau g,\phi, \tau) = \mu_\alpha(g,\phi,1).
\end{equation*}
We can hence reduce the problem to the special case where $\tau =1$.
Set $v= (4\pi)^{-m/4}e^{-f/2}$. This yields
\begin{equation*}
\sW_\alpha(g,\phi,v,1) =\int_M v\Big(Rv -4\Lap v - \alpha\abs{\D\phi}^2v -2v\log v
-\tfrac{mv}{2}\log(4\pi)-mv\Big)dV,
\end{equation*}
and hence $\mu_\alpha(g,\phi,1) = \inf\{\sW_\alpha(g,\phi,v,1) 
\;|\; \int_M  v^2 dV =1\}$ is the smallest eigenvalue of $L(v) = -4\Lap v +
(R-\alpha\abs{\D\phi}^2 -\frac{m}{2}\log(4\pi)-m)v -2v\log v$ and
$v$ is a corresponding normalized eigenvector. As in Section 3, a
unique smooth positive normalized eigenvector $v_{min}$ exists (cf.
Rothaus \cite{Rothaus:Schroedinger} or List \cite{List:diss}) and we
get the following.
\begin{prop}\label{5.prop2}
Let $(g(t),\phi(t))$ solve $(RH)_\alpha$ for a constant $\alpha > 0$
and let $\dt \tau = -1$. Then $\mu_\alpha(g,\phi,\tau)$ is monotone
non-decreasing in time. Moreover, it is constant if and only if
\begin{equation}\label{5.eq10}
\left\{\begin{aligned}0 &= \Rc{}-\alpha\D\phi\otimes\D\phi
+\Hess(f)-\tfrac{g}{2\tau},\\%
0 &= \tau_g \phi -\scal{\D\phi,\D f}. \end{aligned}\right.
\end{equation}
for the minimizer $f$ that corresponds to $v_{min}$. As always, the
monotonicity result stays true if we allow $\alpha(t)$ to be a
positive non-increasing function instead of a constant.
\end{prop}
\begin{proof}
The proof is completely analogous to the proof of Proposition
\ref{1.prop3}, using the monotonicity of $\sW_\alpha$ instead on the
monotonicity of $\sF_\alpha$.
\end{proof}

\subsection{Non-existence of nontrivial breathers}
Breathers correspond to periodic solutions modulo
diffeomorphisms and scaling, generalizing the notion of solitons
defined in Definition \ref{4.def1}.
\begin{defn}\label{5.def3}
A solution $(g(t),\phi(t))_{t\in[0,T)}$ of\/ $(RH)_\alpha$ is called
a \emph{breather} if there exists $t_1,t_2\in[0,T)$, $t_1<t_2$, a
diffeomorphism $\psi:M\to M$ and a constant $c\in\RR_{+}$ such that
\begin{equation}\label{5.eq11}
\left\{\begin{aligned}g(t_2)&=c\,\psi^*g(t_1),\\%
\phi(t_2)&=\psi^*\phi(t_1).\end{aligned}\right.
\end{equation}
The cases $c<1$, $c=1$ and $c>1$ correspond to \emph{shrinking},
\emph{steady} and \emph{expanding} breathers.
\end{defn}
\begin{thm}\label{5.thm4}
Let $M$ and $N$ be closed and let $(g(t),\phi(t))_{t\in[0,T)}$ be a
solution of\/ $(RH)_\alpha$ with $\alpha(t)\equiv\alpha$.
\begin{itemize}
\item[i)] If this solution is a steady breather, then it necessarily is a
steady gradient soliton. Moreover, $\phi(t)$ is harmonic
and $\Rc{}=\alpha\D\phi\otimes\D\phi$, i.e.~the solution is
stationary.%
\item[ii)] If the solution is an expanding breather, then it necessarily
is an expanding gradient soliton. Again $\phi(t)$ must be harmonic
(and thus stationary, $\phi(t)=\phi(0)$), while $g(t)$ changes only
by scaling.%
\item[iii)] If the solution is a shrinking breather, then it has to be a
shrinking gradient soliton.
\end{itemize}
If we assume in addition that $\dim M=2$ or that $(M,g(0))$ is
Einstein, then in the first two cases above, $\phi(t)$ is not only
harmonic but also conformal, hence a minimal branched immersion,
provided that it is non-constant.
\end{thm}
\begin{proof}
This is an application of the monotonicity results for
$\lambda_\alpha(g,\phi)$ from Proposition \ref{1.prop3} and for
$\mu_\alpha(g,\phi,\tau)$ from Proposition \ref{5.prop2}. Since the
proof is very similar to the Ricci flow case solved by Perelman in
\cite{Perelman:entropy}, we closely follow the notes from Kleiner
and Lott \cite{KleinerLott} on Perelman's paper.
\begin{itemize}
\item[i)] Assume $(g(t),\phi(t))_{t\in[0,T)}$ is a steady breather.
Then there exist two times $t_1$, $t_2$ such that (\ref{5.eq11})
holds with $c=1$. From diffeomorphism invariance of
$\lambda_\alpha(g,\phi)$ defined in (\ref{1.eq13}), we obtain
$\lambda_\alpha(g,\phi)(t_1)=\lambda_\alpha(g,\phi)(t_2)$. From
Proposition \ref{1.prop3}, we get condition (\ref{1.eq14}) on
$[t_1,t_2]$, which means that $(g(t),\phi(t))$ must be a gradient
steady soliton according to Lemma \ref{4.lemma2} and uniqueness of
solutions. Moreover, the minimizer $f=-2\log v_{min}$ which realizes
$\lambda_\alpha(g,\phi)$ is the soliton potential. From
$(-4\Lap+R-\alpha \abs{\D\phi}^2)v_{min}
=\lambda_\alpha(g,\phi)v_{min}=:\lambda_\alpha v_{min}$, we obtain
\begin{equation}\label{5.eq12}
2\Lap f -\abs{\D f}^2+R-\alpha\abs{\D\phi}^2=\lambda_\alpha.
\end{equation}
Since $(g(t),\phi(t))$ is a steady soliton, (\ref{4.eq4}) holds with
$\sigma=0$. Plugging this into (\ref{5.eq12}) yields $\Lap f-\abs{\D
f}^2=\lambda_\alpha$, and we obtain from $\int_Me^{-f}dV=1$
\begin{equation*}
\lambda_\alpha=\int_M\lambda_\alpha e^{-f}dV = \int_M(\Lap f-\abs{\D
f}^2)e^{-f}dV = -\int_M\Lap(e^{-f})dV =0,
\end{equation*}
i.e.~$\Lap f=\abs{\D f}^2$. Another integration yields
\begin{equation*}
\int_M \abs{\D f}^2dV=\int_M\Lap f\;dV=0,
\end{equation*}
and thus $\D f\equiv 0$, $\Hess(f)\equiv 0$ on $M\times[0,T)$ and
(\ref{1.eq14}) becomes
\begin{equation*}
\left\{\begin{aligned}0 &= \Rc{}-\alpha\D\phi\otimes\D\phi,\\%
0 &= \tau_g \phi. \end{aligned}\right.
\end{equation*}
In particular, $\phi(t)$ is harmonic and $(g(t),\phi(t))$ is
stationary.%
\item[ii)] The proof here is analogous to the case of steady
breathers, but we first need to construct a scaling invariant
version of $\lambda_\alpha(g,\phi)$. We define
\begin{equation}\label{5.eq13}
\bar{\lambda}_\alpha(g,\phi):=\lambda_\alpha(g,\phi)\bigg(\int_M
dV_g\bigg)^{\!2/m}.
\end{equation}
This quantity is invariant under rescaling $\tilde{g}=c g$. A proof
of this fact is given in the appendix of the author's dissertation
\cite{Muller:diss}. Moreover, we claim that at times where
$\bar{\lambda}_\alpha(t):=\bar{\lambda}_\alpha(g,\phi)(t)\leq 0$, we
have $\dt\bar{\lambda}_\alpha(t)\geq 0$. Indeed, note that 
$\bar{f}=\log V(t)$ satisfies $\int_M
e^{-\bar{f}}dV=1$ and is thus an admissible test function in the
definition of $\lambda_\alpha$, hence
\begin{equation}\label{5.eq15}
\lambda_\alpha(g,\phi)\leq \sF_\alpha(g,\phi,\log V(t))=\int_M
S\,e^{-\log V(t)}dV=V(t)^{-1}\int_M S\; dV.
\end{equation}
With the assumption $\lambda_\alpha(t)\leq 0$, we find
\begin{equation}\label{5.eq17}
\begin{aligned}
\dt\bar{\lambda}_\alpha &\geq 2V^{2/m}\int_M
\Big(\abs{\sS+\Hess(f)-\tfrac{1}{m}(S+\Lap f)g}^2+\alpha \abs{\tau_g
\phi-\scal{\D\phi,\D f}}^2\Big)d\mu\\%
&\quad\,+\tfrac{2}{m}V^{2/m}\bigg(\int_M(S+\Lap f)^2d\mu
-\Big(\int_M(S+\Lap f)d\mu\Big)^2\bigg),
\end{aligned}
\end{equation}
the right hand side being nonnegative by H\"{o}lder's inequality (see again
\cite{Muller:diss} for a more detailed computation). Now, assume that 
$(g(t),\phi(t))$ is an expanding breather. Since
$\bar{\lambda}_\alpha(g,\phi)$ is invariant under diffeomorphisms
and scaling, we have
$\bar{\lambda}_\alpha(t_1)=\bar{\lambda}_\alpha(t_2)$ for the two
times $t_1$, $t_2$ that satisfy (\ref{5.eq11}). Since
$V(t_1)<V(t_2)$, there must be a time $t_0\in[t_1,t_2]$ with $\dt
V(t_0)>0$ and hence with (\ref{5.eq15})
\begin{equation*}
\lambda_\alpha(t_0)\leq V(t_0)^{-1}\int_M S\; dV =-V(t_0)^{-1}\dt
V(t_0) <0.
\end{equation*}
The claim applies and we obtain $\bar{\lambda}_\alpha(t_1)
\leq\bar{\lambda}_\alpha(t_0)<0$ and since
$\bar{\lambda}_\alpha(t_2)=\bar{\lambda}_\alpha(t_1)$, we see that
$\bar{\lambda}_\alpha(t)$ must be a negative constant. Hence, both
lines on the right hand side of (\ref{5.eq17}) have to vanish. This
means that $(S+\Lap f)$ has to be constant in space for all $t$ and
because $\lambda_\alpha(t)=\int_M(S+\Lap f)d\mu$, this constant has to be
$\lambda_\alpha(t)$. From the first line of (\ref{5.eq17}) we obtain
\begin{equation}\label{5.eq18}
\left\{\begin{aligned}0 &= \Rc{}-\alpha\D\phi\otimes\D\phi+\Hess(f)
-\tfrac{\lambda_\alpha}{m}g,\\%
0 &= \tau_g \phi-\scal{\D\phi,\D f}. \end{aligned}\right.
\end{equation}
By Lemma \ref{4.lemma2}, $(g(t),\phi(t))_{t\in[0,T)}$ is an
expanding soliton with potential $f=-2\log v_{min}$. This means that
we can use (\ref{5.eq12}), which implies
\begin{equation}\label{5.eq19}
0= 2\Lap f-\abs{\D f}+S-\lambda_\alpha = 2\Lap f -\abs{\D f}+S-(\Lap
f + S) = \Lap f - \abs{\D f}
\end{equation}
and thus  by integration $\D f \equiv 0$, $\Hess(f)\equiv 0$, as
above. Plugging this into (\ref{5.eq18}), the second equation tells
us that $\phi(t)$ is harmonic and the first equation yields
\begin{equation*}
\dt g = -2\Rc{}+2\alpha\D\phi\otimes\D\phi =
-2\tfrac{\lambda_\alpha}{m}g,
\end{equation*}
i.e.~$(M,g(t))$ simply expands without changing its
shape.%
\item[iii)] If $(g(t),\phi(t))$ is a shrinking breather, there exist
$t_1$, $t_2$ and $c<1$ which such that (\ref{5.eq11}) is satisfied.
We define
\begin{equation*}
\tau_0:= \frac{t_2-c\,t_1}{1-c}>t_2, \qquad\text{and}\qquad
\tau(t)=\tau_0-t.
\end{equation*}
Note that $\tau(t)$ is always positive on $[t_1,t_2]$. Moreover,
$c=(\tau_0-t_2)/(\tau_0-t_1)=\tau(t_2)/\tau(t_1)$. Then, from the
scaling behavior of $\mu_\alpha(g,\phi,\tau)$ and diffeomorphism
invariance we obtain
\begin{equation}\label{5.eq20}
\begin{aligned}
\mu_\alpha(g(t_2),\phi(t_2),\tau(t_2))&=
\mu_\alpha(c\,\psi^*g(t_1),\psi^*\phi(t_1),c\tau(t_1))\\
&=\mu_\alpha(\psi^*g(t_1),\psi^*\phi(t_1),\tau(t_1))\\
&=\mu_\alpha(g(t_1),\phi(t_1),\tau(t_1)).
\end{aligned}
\end{equation}
By the equality case of the monotonicity result in Proposition
\ref{5.prop2}, $(g(t),\phi(t))$ must satisfy (\ref{5.eq10}) and
according to Lemma \ref{4.lemma2} thus has to be a gradient
shrinking soliton.
\end{itemize}
It remains to prove the additional statement in the cases where
$\dim M=2$ or $(M,g(0))$ is Einstein. If $(g(t),\phi(t))$ is a
steady or expanding breather, we have seen that $\dt g= cg$. In
particular, if $(M,g(t))$ is Einstein at $t=0$ it remains Einstein
under the flow. Moreover, since $\Rc{}= \tfrac{R}{m}g$ in these two
cases, we get
\begin{equation*}
(\phi^*\gamma)_{ij}=\D_i\phi\D_j\phi=\tfrac{1}{2\alpha}\big(\dt
g_{ij}+2R_{ij}\big)=\tfrac{1}{2\alpha}\big(2\tfrac{R}{m}+c\big)g_{ij},
\end{equation*}
i.e.~$\phi$ is conformal. It is a well-known fact that conformal
harmonic maps have to be minimal branched immersions
(cf.~Hartman-Wintner \cite{HartmanWintner}).
\end{proof}

\section{Reduced volume and non-collapsing theorem}
Let us briefly restate the main result from our previous article
\cite{Muller:MonotoneVolumes} about the monotonicity of reduced
volumes for flows of the form $\dt g_{ij}=-2S_{ij}$, where $S_{ij}$ 
is a symmetric tensor with trace $S=g^{ij}S_{ij}$. (The resuts can also be
found in the authors thesis \cite{Muller:diss}.)

\subsection{Monotonicity of backwards reduced volume}
In order to define the backwards reduced distance and volume, we
need a backwards time $\tau(t)$ with $\dt\tau(t) = -1$. Without loss
of generality, one may assume (possibly after a time shift) that
$\tau = -t$.
\begin{defn}\label{6.def2}
Assume $\dtau g_{ij}=2S_{ij}$ has a solution for $\tau \in [0,\bar{\tau}]$ and
$0\leq \tau_1<\tau_2\leq \bar{\tau}$, we define the $\sL_b$-length
of a curve $\eta:[\tau_0,\tau_1]\to M$ by
\begin{equation*}
\sL_b(\eta) := \int_{\tau_0}^{\tau_1}\sqrt{\tau}\left(S(\eta(\tau))
+ \Abs{\tfrac{d}{d\tau}\eta(\tau)}^2 \right)d\tau.
\end{equation*}
Fix the point $p\in M$ and $\tau_0=0$ and define the backwards
reduced distance by
\begin{equation}\label{6.eq4}
\ell_b(q,\tau_1):= \inf_{\eta \in
\Gamma}\left\{\frac{1}{2\sqrt{\tau_1}}\int_0^{\tau_1}\sqrt{\tau}
\left(S+\Abs{\tfrac{d}{d\tau}\eta}^2\right)d\tau\right\},
\end{equation}
where $\Gamma= \{\eta:[0,\tau_1]\to M \mid
\eta(0)=p,\,\eta(\tau_1)=q\}$. The backwards reduced volume is
defined by
\begin{equation}\label{6.eq5}
\tilde{V}_b(\tau) := \int_M
(4\pi\tau)^{-m/2}e^{-\ell_b(q,\tau)}dV(q).
\end{equation}
\end{defn}
The following is proved in \cite[Theorem
1.4]{Muller:MonotoneVolumes}.
\begin{thm}\label{6.thm4}%
Suppose that $g(t)$ evolves by $\dt g_{ij}=-2S_{ij}$ and the quantity
\begin{equation}\label{6.eq6}
\begin{aligned}
\sD(\sS,X) &:= \dt S-\Lap S -2\Abs{S_{ij}}^2
+4(\D_i S_{ij})X_j -2(\D_j S)X_j\\%
&\phantom{:}\quad + 2R_{ij}X_iX_j - 2S_{ij}X_iX_j,
\end{aligned}
\end{equation} is nonnegative for all vector fields $X \in \Gamma(TM)$
and all times $t$ for which the flow exists. Then the backwards
reduced volume $\tilde{V}_b(\tau)$ is non-increasing in $\tau$,
i.e.~non-decreasing in $t$.
\end{thm}
In our case where $S_{ij}$ is given by $R_{ij}- \alpha\D_i\phi\D_j\phi$, the evolution
equation (\ref{2.eq16}) for $S_{ij}$ together with $4(\D_i
S_{ij})X_j - 2(\D_j S)X_j = -4\alpha\, \tau_g \phi \D_j\phi X_j$
yields
\begin{equation*}
\sD(S_{ij},X)=2\alpha \abs{\tau_g\phi-\D_X\phi}^2
-\dot{\alpha}\abs{\D\phi}^2
\end{equation*}
for all $X$ on $M$. This means, $\sD(S_{ij},X)\geq 0$ is satisfied
for the $(RH)_\alpha$ flow with a positive non-increasing coupling
function $\alpha(t)$ and the monotonicity of the reduced volume
holds.

\subsection{No local collapsing theorem}
We have seen in Section 6 that the metrics $g(t)$ along the
$(RH)_\alpha$ flow are uniformly equivalent as long as the curvature
on $M$ stays uniformly bounded. But a-priori, it could happen that at a
singularity (i.e.~when $\Rm{}$ blows up) the solution collapses
geometrically in the following sense.
\begin{defn}
Let $(g(t),\phi(t))_{t\in[0,T)}$ be a maximal solution of\/
$(RH)_\alpha$, or more generally of any flow of the form $\dt
g_{ij}=-2S_{ij}$. We say that this solution is locally collapsing at
time $T$, if there is a sequence of times $t_k\nearrow T$ and a
sequence of balls $B_k:=B_{g(t_k)}(x_k,r_k)$ at time $t_k$, such
that the following holds. The ratio $r_k^2/t_k$ is bounded, the
curvature satisfies $\abs{\Rm{}}\leq r_k^{-2}$ on the parabolic
neighborhood $B_k\times[t_k-r_k^2,t_k]$ and $r_k^{-m}\vol(B_k)\to 0$
as $k\to\infty$.
\end{defn}
Using the monotonicity of the reduced volume, we obtain the
following result.
\begin{thm}\label{6.thm13}
Let $(g(t),\phi(t))$ be a solution of\/ $(RH)_\alpha$ with
non-increasing $\alpha(t)\in[\underaccent{\bar}\alpha,\bar\alpha]$,
$0<\underaccent{\bar}\alpha\leq\bar\alpha<\infty$ on a finite time
interval $[0,T)$. Then this solution is not locally collapsing at
$T$.
\end{thm}
The only ingredients of the proof are the interior
gradient estimates from Corollary \ref{app.prop5} and the
monotonicity of the backwards reduced volume stated above. Hence,
every flow $\dt g=-2\sS$ that satisfies the assumption of Theorem
\ref{6.thm4} and some interior estimates for $\sS$, $\D S$ in the
spirit of Corollary \ref{app.prop5} will also satisfy the
non-collapsing result. For the $(RH)_\alpha$ flow, it is possible to
obtain a slightly stronger result using the monotonicity of
$\mu_\alpha(g,\phi,\tau)$ from Section 7 instead of the monotonicity
of the backwards reduced volume. In the special case $N\subseteq\RR$, 
this can be found in List's dissertation \cite[Section 7]{List:diss}. 
The proof in the case of $(RH)_\alpha$ is analogous. However, the 
result here is more general in the sense that it may be adopted to 
other flows $\dt g=-2\sS$ in the way explained above.
\begin{proof}
The proof follows Perelman's results for the Ricci
flow \cite{Perelman:entropy} very closely, see also the notes on his
paper by Kleiner and Lott \cite{KleinerLott} and the book by Morgan
and Tian \cite{MorganTian}. However, we need the more general
results from \cite{Muller:MonotoneVolumes} that also hold for our 
coupled flow system. We only give a sketch.%
\newline

The proof is by contradiction. Assume that there is some sequence of
times $t_k\nearrow T$ and some sequence of balls
$B_k:=B_{g(t_k)}(x_k,r_k)$ at each time $t_k$, such that $r_k^2$ is
bounded, the curvature is bounded by $\abs{\Rm{}}\leq r_k^{-2}$ on
the parabolic neighborhood $B_k\times[t_k-r_k^2,t_k]$ and
$r_k^{-m}\vol(B_k)\to 0$ as $k\to\infty$. Define
$\eps_k:=r_k^{-1}\vol(B_k)^{1/m}$, then $\eps_k\to 0$ for
$k\to\infty$. For each $k$, we set $\tau_k(t)=t_k-t$ and define the
backwards reduced volume $\tilde{V}_k$ using curves going backward
in real time from the base point $(x_k,t_k)$, i.e.~forward in time
$\tau_k$ from $\tau_k=0$. The goal is to estimate the reduced
volumes $\tilde{V}_k(\eps_kr_k^2)$, where $\tau_k=\eps_kr_k^2$
corresponds to the real time $t=t_k-\eps_k r_k^2$, which is very
close to $t_k$ and hence close to $T$.
\paragraph{Claim 1:} $\lim_{k\to\infty}\tilde{V}_k(\eps_kr_k^2)=0$.
\begin{proof}
An $\sL_b$-geodesic $\eta(\tau)$ starting at $\eta(0)=x_k$ is
uniquely defined through its initial vector $v=\lim_{\tau\to
0}2\sqrt{\tau}X=\lim_{\lambda\to0}\tilde{X}$. First, we show that if
$\abs{v}\leq \frac{1}{8}\eps_k^{-1/2}$ with respect to the metric at
$(x_k,t_k)$, then $\eta(\tau)$ does not escape from
$B_k^{1/2}:=B_{g(t_k)}(x_k,r_k/2)$ in time $\tau=\eps_kr_k^2$. Write
$\hat{t}_k=t_k-r_k^2$. Since $\abs{\Rm{}}\leq r_k^{-2}$ on
$B_k\times[\hat{t}_k,t_k]$ by assumption, we obtain from Corollary
\ref{app.prop5}
\begin{equation*}
\abs{\D\phi}^2 \leq \frac{C}{t-\hat{t}_k},\quad\abs{\D^2\phi}^2 \leq
\frac{C}{(t-\hat{t}_k)^2},\quad\abs{\D\Rm{}}^2\leq
\frac{C}{(t-\hat{t}_k)^3},\quad\textrm{on }B_k^{1/2}
\times(\hat{t}_k,t_k),
\end{equation*}
for some constant $C$ independent of $k$. Without loss of
generality, $\eps_k\leq\tfrac{1}{2}$ so that
$t-\hat{t}_k\geq\tfrac{1}{2}r_k^2$ whenever $t\in
[t_k-\eps_kr_k^2,t_k)$. This means that
\begin{equation*}
\abs{\D\phi}^2 \leq Cr_k^{-2},\quad\abs{\D^2\phi} \leq
Cr_k^{-2},\quad\abs{\D\Rm{}}\leq Cr_k^{-3},\quad\textrm{on
}B_k^{1/2}\times[t_k-\eps_kr_k^2,t_k).
\end{equation*}
Together with the assumption $\abs{\Rm{}}\leq r_k^{-2}$, this yields
\begin{equation}\label{6.eq38}
\begin{aligned}
\abs{\sS}&\leq\abs{\Rc{}}+\abs{\D\phi}^2\leq Cr_k^{-2},\\%
\abs{\D S}&\leq \abs{\D R}+\abs{\D\phi}\abs{\D^2\phi}\leq Cr_k^{-3},
\end{aligned}
\end{equation}
on $B_k^{1/2}\times[t_k-\eps_kr_k^2,t_k)$. Plugging this into the
estimate (4.5) of \cite{Muller:MonotoneVolumes}, we get
\begin{equation}\label{6.eq39}
\dl\lvert\tilde{X}\rvert \leq \lambda C\lvert\tilde{X}\rvert
r_k^{-2} + \lambda^2 Cr_k^{-3}\leq C\lvert\tilde{X}\rvert
\eps_k^{1/2}r_k^{-1}+C\eps_kr_k^{-1},
\end{equation}
for $\lambda=\sqrt{\tau}\leq\sqrt{\eps_kr_k^2}=\eps_k^{1/2}r_k$.
Since $\lvert\tilde{X}(0)\rvert=\abs{v}\leq\frac{1}{8}\eps_k^{-1/2}$
we obtain the estimate $\lvert\tilde{X}(\lambda)\rvert
\leq\frac{1}{4}\eps_k^{-1/2}$ for all $\tau\in[0,\eps_kr_k^2]$ if
$k$ is large enough, i.e.~$\eps_k$ small enough. With an
integration, we find
\begin{equation*}
\int_0^{\eps_kr_k^2}\abs{X(\tau)}d\tau = \int_0^{\sqrt{\eps_k}\,r_k}
\lvert\tilde{X}(\lambda)\rvert d\lambda
\leq\int_0^{\sqrt{\eps_k}\,r_k} \tfrac{1}{4}\eps_k^{-1/2}d\lambda
\leq \tfrac{1}{4}r_k.
\end{equation*}
Since the metrics $g(\tau=0)$ and $g(\tau=\eps_kr_k^2)$ are close to
each other, the length of the curve $\eta$ measured with respect to
$g(\tau=0)=g(t_k)$ will be at most $r_k/2$ for large enough $k$.
This means that indeed
\begin{equation}\label{6.eq40}
(\eta(\tau),t_k-\tau)\in
B_k^{1/2}\times[t_k-\eps_kr_k^2,t_k),\qquad\forall 0<\tau\leq
\eps_kr_k^2.
\end{equation}
With the bounds from (\ref{6.eq38}) and the lower bound in
\cite[Lemma 4.1]{Muller:MonotoneVolumes}, we obtain
\begin{equation*}
\sL_b(\eta)\geq -Cr_k^{-2}(\eps_kr_k^2)^{3/2} =
-C\eps_k^{3/2}r_k,\quad\textrm{i.e.}\quad\ell_b(q,\eps_kr_k^2)\geq
-C\eps_k
\end{equation*}
Thus, the contribution to the reduced volume
$\tilde{V}_k(\eps_kr_k^2)$ coming from $\sL_b$-geodesics with
initial vector $\abs{v}\leq\frac{1}{8}\eps_k^{-1/2}$ is bounded
above for large $k$ by
\begin{equation*}
\int_{B_k^{1/2}}(4\pi\eps_kr_k^2)^{-m/2}e^{C\eps_k}dV \leq
C\eps_k^{-m/2}r_k^{-m}\vol(B_k^{1/2})\leq C\eps_k^{m/2}\to
0\quad(k\to\infty).
\end{equation*}
Next, we estimate the contribution of geodesics with large initial
vector $\abs{v}> \frac{1}{8}\eps_k^{-1/2}$ to the reduced volume
$\tilde{V}_k(\eps_kr_k^2)$. Note that we can write the reduced 
volume with base point $(x_k,t_k)$ as
\begin{equation*}
\tilde{V}_k(\tau_1)=\int_M
(4\pi\tau_1)^{-m/2}e^{-\ell(q,\tau_1)}dV(q) =
\int_{\Omega(\tau_1,k)}
(4\pi\tau_1)^{-m/2}e^{-\ell(\Lexp(v),\tau_1)}J(v,\tau_1) dv.
\end{equation*}
Here, $\Lexp$ is the $\sL_b$-exponential map defined in
\cite{Muller:MonotoneVolumes}, taking $v$ to $\eta(\tau_1)$ with
$\eta$ being the $\sL_b$-geodesic with initial vector $v$,
$J(v,\tau_1)=\det d(\Lexp)$ denotes the Jacobian of $\Lexp$ and
$\Omega(\tau_1,k)\subset T_{x_k}M$ is a set which is mapped
bijectively to $M$ up to a set of measure zero under the map
$\Lexp$. In \cite{Muller:diss}, we prove that the integrand
\begin{equation}\label{6.eq41}
f(v,\tau_1):=(4\pi\tau_1)^{-m/2}e^{-\ell(\Lexp(v),\tau_1)}J(v,\tau_1)
\end{equation}
is non-increasing in $\tau_1$ for fixed $v$ and has the limit
$\lim_{\tau_1\to 0}f(v,\tau_1)=\pi^{-m/2}e^{-\abs{v}^2}$. Together
with $\Omega(\tau')\subset\Omega(\tau)$ for $\tau\leq\tau'$ this
yields an alternative proof of the monotonicity of the
reduced volumes obtained in \cite{Muller:MonotoneVolumes}. Moreover,
it implies that the contribution to the reduced volume
$\tilde{V}_k(\eps_kr_k^2)$ coming from $\sL_b$-geodesics with
initial vector $\abs{v}>\frac{1}{8}\eps_k^{-1/2}$ can be bounded by
\begin{equation}\label{6.eq45}
\int_{\abs{v}>\frac{1}{8}\eps_k^{-1/2}}\pi^{-m/2}e^{-\abs{v}^2}dv
\leq Ce^{-\frac{1}{64\eps_k}} \to 0\quad(k\to\infty),
\end{equation}
which completes the proof of Claim 1.
\end{proof}
\paragraph{Claim 2:} $\tilde{V}_k(t_k)$ is bounded below away from
zero.
\begin{proof}
Let us remark that $\tau=t_k$ corresponds to real time $t=0$. We
assume that $k$ is large enough, so that $t_k\geq T/2$. The idea
behind the proof is to go from $(x_k,t_k)$ to some point $q_k$ at
the real time $T/2$ (i.e.~$\tau=t_k-T/2$) for which the reduced
$\sL_b$-distance $\ell_b(q_k,t_k-T/2)$ is small. From the upper
bound on $L_b$ from \cite[Lemma 4.1]{Muller:MonotoneVolumes}, we see
that for small $\tau$ it is possible to find a point $q_k(\tau)$
such that $\ell_b(q_k(\tau),\tau)\leq \frac{m}{2}$. On the other
hand, combining the evolution equations for $\dtau \ell_b$ and $\Lap
\ell_b$, we obtain
\begin{equation}\label{6.eq46}
\dtau\big|_{\tau=\tau_1} \ell_b + \Lap \ell_b \leq
-\tfrac{1}{\tau_1}\ell_b + \tfrac{m}{2\tau_1}
\end{equation}
(in the barrier sense) and hence for the minimum of
$\ell_{min}(\tau)=\min_{q\in M}\ell_b(q,\tau)$
\begin{equation}\label{6.eq47}
\dtau\big|_{\tau=\tau_1} \ell_{min} \leq
-\tfrac{1}{\tau_1}\ell_{min} + \tfrac{m}{2\tau_1}
\end{equation}
in the sense of difference quotients. The latter is obtained by
applying the maximum principle to a smooth barrier. The inequality
(\ref{6.eq47}) shows that there is some point $q_k(\tau)$ with
$\ell_b(q_k(\tau),\tau)\leq \frac{m}{2}$ for every $\tau$. As
mentioned above, we choose $q_k$ at the real time $T/2$ with
$\ell_b(q_k,t_k-T/2)\leq \frac{m}{2}$. Let $\eta:[0,t_k-T/2]\to M$
be an $\sL_b$-geodesic realizing this length. Moreover, let
$\eta_p:[t_k-T/2,t_k]\to M$ be $g(t=0)$-geodesics (i.e. a
$g(\tau=t_k)$-geodesic) from $q_k$ at time $\tau=t_k-T/2$ to $p\in
B^{q_k}:=B_{g(\tau=t_k)}(q_k,1)=B_{g(t=0)}(q_k,1)$ at time
$\tau=t_k$. Since $\abs{\Rm{}}$ is uniformly bounded for
$t\in[0,T/2]$, we get a uniform bound for $S$ along this family of curves. 
Since all the metrics $g(\tau)$ with $\tau\in[t_k-T/2,t_k]$ are uniformly
equivalent, we get an uniform upper bound for the $\sL_b$-length of
all $\eta_p$. From this, we see that the concatenations
$(\eta\!\smile\!\eta_p): [0,t_k]\to M$ connecting $x_k$ to $p\in
B^{q_k}$ have uniformly bounded $\sL_b$-length, independent of $p$
and $k$. This gives a uniform bound $\ell_b(p,t_k)\leq C$, for all
$p\in B^{q_k}$ and $k\in\NN$ large enough. We can then estimate
\begin{equation*}
\tilde{V}_k(t_k)=\int_M (4\pi t_k)^{-m/2}e^{-\ell_b(q,t_k)}dV(q)
\geq \int_{B^{q_k}} (4\pi t_k)^{-m/2}e^{-C}dV \geq C\inf_{q_k\in
M}\vol({B^{q_k}}),
\end{equation*}
which is bounded below away from zero, independently of $k$. This
proves Claim 2.
\end{proof}
Since the backwards reduced volumes $\tilde{V}_k$ are non-increasing
in $\tau$ (i.e.~non-decreasing in real time $t$) according to
Theorem \ref{6.thm4}, we obtain $\tilde{V}_k(t_k)\leq
\tilde{V}_k(\eps_kr_k^2)$ for $k$ large enough. But since
$\tilde{V}_k(t_k)$ is bounded below away from zero by Claim 2 while
$\tilde{V}_k(\eps_kr_k^2)$ converges to zero with $k\to\infty$ by
Claim 1, we obtain the desired contradiction that proves the
theorem.
\end{proof}
\begin{appendix}

\section{Commutator identities}
It is well-known that a $(p,q)$-tensor $B$ (i.e.~a smooth
section of the bundle $(T^*M)^{\otimes p}\otimes(TM)^{\otimes q}$) 
satisfies the following commutator identity in local coordinates 
$(x^1,\ldots,x^n)$ induced by a chart $\phi: U \to \RR^n$, $U\subseteq M$,
\begin{equation}\label{app.eq7}
[\D_i,\D_j]B^{k_1\dots k_q}_{\ell_1\dots\ell_p} = \sum_{r=1}^q
R_{ijm}^{k_r}B^{k_1\dots k_{r-1}m k_{r+1}\dots
k_q}_{\ell_1\dots\ell_p} + \sum_{s=1}^p R_{ij\ell_s m}B^{k_1\dots
k_q}_{\ell_1\dots\ell_{s-1}m\ell_{s+1}\dots\ell_p}.
\end{equation}
Now, assume that we are given Levi-Civita connections for all
$(p,q)$ tensors over $(M,g)$ and over $(N,\gamma)$. For a map $\phi:(M,g)\to
(N,\gamma)$, there is a canonical notion of pull-back bundle
$\phi^*TN$ over $M$ with sections $\phi^*V = V\circ\phi$ for
$V\in\Gamma(TN)$. The Levi-Civita connection $\D^{TN}$ on $TN$ also
induces a connection $\D^{\phi^*TN}$ on this pull-back bundle via
\begin{equation*}
\D^{\phi^*TN}_X\phi^*V = \phi^*\big(\D^{TN}_{\phi_*X}V\big), \qquad
X\in\Gamma(TM),\; V\in\Gamma(TN).
\end{equation*}
We obtain connections on all product bundles over $M$ with
factors $TM$, $T^*M$, $\phi^*TN$ and $\phi^*T^*N$ via the product
rule and compatibility with contractions. Take coordinates $x^k$ 
on $M$, $k=1,\ldots,m=\dim M$, and $y^\mu$ on
$N$, $\mu=1,\ldots,n=\dim N$, and write $\partial_k$ for
$\tfrac{\partial}{\partial x^k}$ and $\partial_\mu$ for
$\tfrac{\partial}{\partial y^\mu}$. We get $\D_i\D_jV_\kappa-\D_j\D_iV_\kappa=R_{ij\kappa\lambda}V_\lambda$ with
\begin{align*}
R_{ij\kappa\lambda}(x)&=\big\langle\Rm{\partial_i,
\partial_j}\phi^*\big(\partial_\lambda\big),
\phi^*\big(\partial_\kappa\big)\big\rangle_{\phi^*TN}(x)\\%
&=\Scal{\NRm{}\big(\phi_*\partial_i, \phi_*\partial_j\big)
\partial_\lambda, \partial_\kappa}_{TN}(\phi(x))\\%
&= \hN\!R_{\mu\nu\kappa\lambda}(\phi(x))
\D_i\phi^\mu(x)\D_j\phi^\nu(x),
\end{align*}
where we used $\phi_*\partial_i =\D_i\phi^\mu\,\partial_\mu$. This 
allows to extend (\ref{app.eq7}) to mixed tensors, for example
\begin{equation}\label{app.eq8}
[\D_i,\D_j]B^{k\kappa}_{\ell\lambda} =
R_{ijp}^kB^{p\kappa}_{\ell\lambda} + R_{ij\ell
p}B^{k\kappa}_{p\lambda} + R_{ij\varrho}^\kappa
B^{k\varrho}_{\ell\lambda} + R_{ij\lambda\varrho}
B^{k\kappa}_{\ell\varrho}.
\end{equation}
The standard example that will be used quite often is the following.
The derivative $\D\phi$ of $\phi:M\to N$ is a section of
$T^*M\otimes\phi^*TN$. Thus, the intrinsic second order derivative
is built with the connection on this bundle, i.e.~$\D_i\D_j\phi^\lambda = \partial_i\partial_j\phi^\lambda -
\Gamma_{ij}^k\partial_k\phi^\lambda + \hN\Gamma_{\mu\nu}^\lambda
\partial_i\phi^\mu\partial_j\phi^\nu$ and similar for higher derivatives. 
Using (\ref{app.eq8}), we obtain
\begin{equation}\label{app.eq10}
\D_i\D_j\D_\ell\phi^\beta-\D_j\D_i\D_\ell\phi^\beta = R_{ij\ell p}\D_p\phi^\beta +
\hN\!R_{\mu\nu\lambda}^\beta\D_\ell\phi^\lambda\D_i\phi^\mu\D_j\phi^\nu.
\end{equation}
There is also a different way to obtain these formulas, which is
especially useful when $\phi$ is evolving and we also want to
include time derivatives. We learned this from \cite{Lamm:diss}.
Here, we interpret $\D^k\phi$ as a $k$-linear $TN$-valued map along
$\phi\in C^\infty(M,N)$ rather than as a section in $(T^*M)^{\otimes
k}\otimes \phi^*TN$. Leting $\omega$ be any such $k$-linear $TN$-valued 
map along $\phi$, i.e.~$\omega(x):(T_xM)^{\times k}\to T_{\phi(x)}N$,
the covariant derivative $\D\omega$ is a $(k+1)$-linear $TN$-valued
map along $\phi$ etc. The curvature tensor $\!\phantom{l}^k\Rm{}$ for 
$\omega$ can then be computed by
\begin{equation}\label{app.eq12}
\begin{aligned}
(\!\phantom{l}^k\!\Rm{X,Y}\omega)(X_1,\ldots,X_k)
&=\NRm{\D\phi(X),\D\phi(Y)}\omega(X_1,\ldots,X_k)\\%
&\quad\,-\sum_{s=1}^k \omega(X_1,\ldots,\Rm{X,Y}X_s,\ldots,X_k).
\end{aligned}
\end{equation}
Of course, this agrees with the definition above, where we used the
bundle interpretation. Now, if $\phi$ is time-dependent, we simply
interpret it as a map $\tilde{\phi}:M\times I\to N$ and interpret
$\D^k\phi$ as $k$-linear $TN$-valued maps on $M\times I$ along
$\tilde{\phi}$. The formalism stays exactly the same.%
\newline

Note that $\dt$ induces a covariant time derivative $\D_t$ (on all
bundles over $M\times I$) that agrees with $\dt$ for time-dependent
functions. Choose coordinates $x^i$ for $M$ with
\begin{equation}\label{app.eq13}
\D_t(\dt) = \D_i(\partial_j) = \D_t(\partial_i) = \D_i(\dt) = 0,
\quad \forall i,j=1,\ldots,m
\end{equation}
at some base point $(p,t)$ in $M\times I$. Then, using
(\ref{app.eq12}), we obtain for $\omega=\D\phi$
\begin{equation}\label{app.eq14}
\begin{aligned}
\D_t(\D_i\D_j\phi) &= \D_t((\D_i\omega)(\partial_j)) =
(\D_t\D_i\omega)(\partial_j) = (\D_i\D_t\omega +
\!\!\phantom{l}^1\!\Rm{\dt,\partial_i}\omega)(\partial_j)\\%
&=\D_i((\D_t\omega)(\partial_j))+\NRm{\dt\phi,\D_i\phi}\omega(\partial_j)
- \omega(\!\phantom{l}^{M\times I}\!\Rm{\dt,\partial_i}\partial_j)\\%
&=\D_i\D_j\dt\phi + \NRm{\dt\phi,\D_i\phi}\D_j\phi.
\end{aligned}
\end{equation}
\begin{rem}
If we also vary the metric $g$ on $M$ in time, we will get an
additional term from the evolution of $\D_i\D_j$, namely
$-(\dt\Gamma_{ij}^k)\D_k\phi$. Note that $\dt\Gamma$ is a tensor,
while $\Gamma$ itself is not.
\end{rem}
\end{appendix}

\makeatletter
\def\@listi{%
  \itemsep=0pt
  \parsep=1pt
  \topsep=1pt}
\makeatother
{\fontsize{10}{11}\selectfont

\vspace{10mm}

Reto M\"uller\\
{\sc Scuola Normale Superiore di Pisa, 56126 Pisa, Italy}
\end{document}